\documentclass[USenglish]{article}
\usepackage[utf8]{inputenc}

\usepackage{amssymb,amsmath,amsthm,amscd,epsf,latexsym,verbatim,graphicx,amsfonts,hyperref,epstopdf,color}
\input epsf.tex
\usepackage{graphicx}
\usepackage{palatino, url, multicol}
\usepackage{subfig}
\usepackage{geometry}                
\theoremstyle{theorem}
\newtheorem{theorem}{Theorem}[section]
\newtheorem{conjecture}[theorem]{Conjecture}
\newtheorem{lemma}[theorem]{Lemma}
\newtheorem{proposition}[theorem]{Proposition}
\newtheorem{corollary}[theorem]{Corollary}

\theoremstyle{definition}

\theoremstyle{definition}

\theoremstyle{definition}
\newtheorem{remark}[theorem]{Remark}
\theoremstyle{definition}
\newtheorem{example}[theorem]{Example}
\theoremstyle{definition}
\newtheorem{definition}[theorem]{Definition}
\theoremstyle{definition}

\theoremstyle{definition}

\title{Negative refraction and tiling billiards}
\author{Diana Davis, Kelsey DiPietro, Jenny Rustad, Alexander St Laurent}

\begin{document}
\maketitle


\begin{abstract}
We introduce a new dynamical system that we call \emph{tiling billiards}, where trajectories refract through planar tilings. This system is motivated by a recent discovery of physical substances with negative indices of refraction. We investigate several special cases where the planar tiling is created by dividing the plane by lines, and we describe the results of computer experiments.
\end{abstract}

\section{Introduction}


\subsection{Negative index of refraction in physics}

When light travels from one medium to another, the ray of light may bend; this bending is known as \emph{refraction}. Snell's Law\footnote{Willebrord Snell stated this result in 1621, but it was known in 984 to Islamic scholar Abu Said al-Ala Ibn Sahl \cite{kwan}.} describes the change in angle of the ray, relating it to the indices of refraction of the two media. If $\theta_1$ and $\theta_2$ are the angles between the ray and the normal to the boundary in each medium, and $k_1$ and $k_2$ are the indices of refraction of the two media,

$$\frac{\sin\theta_1}{\sin\theta_2}=\frac{k_2}{k_1}.$$ This ratio is called the \emph{refraction coefficient} or the \emph{Snell index}.

In nature, the indices of refraction of all materials are positive. 
Recently, however, physicists have created materials that have negative indices of refraction, called "metamaterials." Articles describing this discovery were published in \emph{Science} \cite{shelby},\cite{smith}, and have been cited thousands of times in the 10 years since.

Even with standard materials with positive indices of refraction, it is possible to do counterintuitive things like make glass ``disappear'' by immersing it in clear oil. 
Since metamaterials are not found in nature, we have even less intuition for their effects.
When light crosses a boundary between a medium with a positive index of refraction and one with a negative index of refraction, the light  bends in the opposite direction as it would if both media had positive indices of refraction.
These new materials may someday be used to create an invisibility shield, improve solar panel technology, or fabricate a perfect lens that would be able to resolve details even smaller than the wavelengths of light used to create the image, see \cite{physicists}.

 The behavior in complex scenarios, such as the propagation of light through materials that alternate between standard materials and metamaterials, is complex and interesting. Motivated by this new discovery in physics, we introduce and study a new billard-type dynamical system. 

\subsection{Tiling billiards}

We formulate a model for light passing between two media with refraction coefficient $-1$, which we call \emph{tiling billiards}, defined as follows:
\begin{enumerate}
\item We consider a partition of the plane into regions (a \emph{tiling}). 
\item We give a refraction rule for a ray of light passing through the tiling: When a ray hits a boundary, it is reflected across a tangent line to the boundary (or across an edge of a straight boundary). If the ray hits a corner, the subsequent trajectory is undefined.
\end{enumerate}

An example of a tiling billiards trajectory is in Figure \ref{tilebill}.

\begin{figure} [!h]
\centering 
\includegraphics[height=90pt]{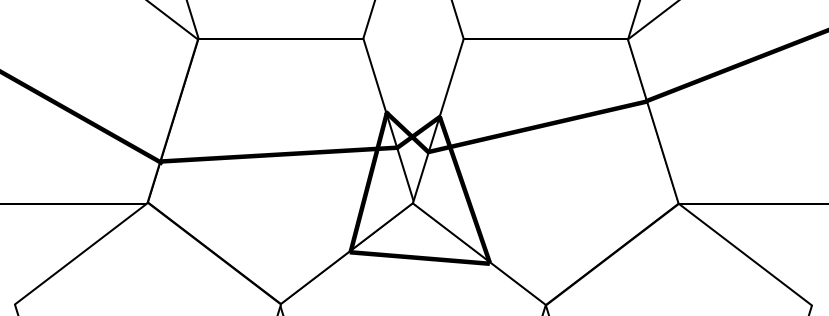}
\caption{A tiling billiards trajectory on a planar tiling of pentagons and rhombi}
\label{tilebill}
\end{figure}

To relate this model to the physical system with standard materials and metamaterials, we consider a two-colorable tilling of the plane, where one color corresponds to a material with index of refraction $k$, and the other color to a material with index $-k$. However, the dynamical system described above need not obey this restriction of two-colorability.

Note that this problem is $1$-dimensional, in the sense that the width of a parallel beam of light is preserved.

\subsection{Dynamics of inner and outer billiards} 

Although the tiling billiards system is new, it has many similarities to the well-studied dynamical systems of inner and outer billiards. In the \emph{inner billiards} dynamical system, we consider a billiard table with a ball bouncing around inside, where the angle of incidence equals the angle of reflection at each bounce. In the \emph{outer billiards} dynamical system, the ball is outside the table, and reflects along a tangent line to the table (or through a vertex of a polygonal table), so that the original distance from the ball to the reflection point is equal to the distance from the reflection point to the image, see \cite{sergei}. In tiling billiards, the ray is reflected across the boundary. Examples of inner and outer billiards are in Figure \ref{three-billiards}.

\begin{figure} [!h]
\centering 
\includegraphics[height=98pt]{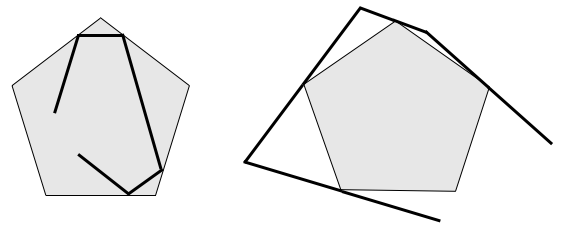} 
\caption{Part of (a) an inner billiards trajectory on a regular pentagon, (b) an outer billiards trajectory on a regular pentagon.}
\label{three-billiards}
\end{figure}

For any billiard system, there are several basic questions to ask. We outline these below, with some answers based on  results from the established literature on inner and outer billiards.

\vspace{1em}\noindent\emph{Are there periodic trajectories?}

In inner billiards on a rational polygon (a polygonal table where every angle is a rational multiple of $\pi$), there are always periodic paths $-$ in fact, Masur \cite{masur} showed that periodic directions are dense. For general polygons, it is not known whether there is always a periodic path, but Schwartz \cite{schwartz} showed that for in inner billiards on a \emph{triangle} whose largest angle is less than $100^{\circ}$, there is always a periodic trajectory. 

Every polygonal outer billiard has a periodic orbit, and for outer billiards on a polygon whose vertices are at lattice points, all of the orbits are periodic, see \cite{kites}, \cite{tabachni} and \cite{vivaldi}.

\vspace{1em}\noindent\emph{Can trajectories escape? }  

For a finite billiard table, an inner billiard trajectory cannot escape.

After the outer billiards system was introduced by B.H. Neumann, it stood as an open question for decades whether there are billiard tables with unbounded orbits, with many negative answers for special cases. Finally, Schwartz \cite{kites} showed that the Penrose kite has unbounded orbits, and Dolgopyat and Fayad \cite{semicircle} showed the same for the half disk.

\vspace{1em}\noindent\emph{Can a trajectory fill a region densely? }
 
In inner billiards on a square table, a trajectory with an irrational slope is dense. It is also possible to construct a table so that some trajectory on that table is dense in one part of the table, but never visits another part of the table at all. McMullen constructed an L-shaped example of this phenomenon \cite{mcmullen}, Galperin constructed examples on polygons with $4$ or more sides \cite{galperin}, and recently Tokarsky asserted that Galperin's triangle construction is flawed, so  the problem of constructing a not-everywhere-dense non-periodic billiard orbit on a triangle is still open \cite{tokarsky}.
 
For outer billiards, Genin \cite{genin} showed that for certain trapezoids, there are trajectories that are dense in  open regions of the plane. 

\vspace{1em}\noindent\emph{ Is the behavior stable under small perturbations of the starting point and direction of the trajectory, and under small perturbations of the billiard table?}

In inner billiards on a square table, the behavior of the trajectory depends on the direction: directions with rational slope yield periodic trajectories, and with irrational slope yield dense trajectories, so a small perturbation changes everything. Similarly, a small perturbation of the polygon taking the square to a nearby quadrilateral with different angles yields very different behavior of the trajectory.  So inner billiards are generally highly sensitive to the starting conditions, and in this sense are not stable. However, if the angles of the polygon are independent over the rationals, then every periodic trajectory is stable, see \cite{billiards}. 

Outer billiards are also unstable under perturbation of the table; for instance, outer billiards on a kite has an unbounded orbit if and only if the kite is irrational, see \cite{kites}. In contrast, outer billiards are quite stable under perturbation of the trajectory direction: if one point yields a periodic trajectory, then that point lies in a polygonal region of points that do the same thing. 

\vspace{1em}\noindent\emph{ Is the behavior invariant under affine transformations?}

Because the inner billiards system uses angles, the behavior is highly sensitive to affine transformations; inner billiards on a rectangular table is equivalent to that of a square table, but inner billiards on a general parallelogram is completely different.

Because the outer billiards system uses ratios of the lengths of collinear segments, the behavior is invariant under affine transformations. So for example, the dynamics on any parallelogram are the same as the dynamics on a square, and the dynamics on every triangle are the same. 

\subsection{Dynamics of tiling billiards}
Tiling billiards is a cross between the inner and outer billiards systems, in the sense that in inner billiards, the ray is reflected across the normal to the boundary; in outer billiards, the ray itself forms a tangent to the boundary and the point is reflected through the point of tangency; and in tiling billiards, the ray is reflected across the tangent to the boundary (or the boundary itself if it is polygonal). An example of tiling billiards is in Figure \ref{penrose}. 

\begin{figure} [!h]
\centering 
\includegraphics[height=70pt]{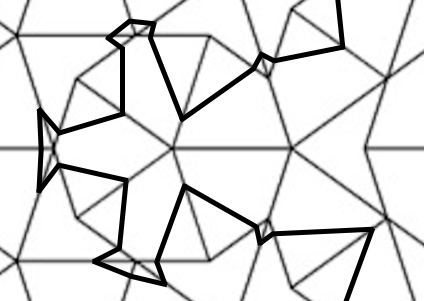} 
\caption{A tiling billiards trajectory (thick) on the Penrose tiling (thin).}
\label{penrose}
\end{figure}

Much of the existing literature on billiards focuses on polygonal billiards. For the tiling billiards system, we study analogous tilings of the plane. In this paper, we study the case where the plane is divided by lines into regions, possibly of infinite area. Such a tiling is always two-colorable (required for the physical application discussed at the beginning), because finitely many lines intersect at any point, and every vertex has an even valence, so the associated adjacency graph of the regions contains no odd cycles, and is thus bipartite. Within this case, we study the three special cases where the plane is divided by a finite number of non-parallel lines (a \emph{line arrangement}), where the resulting tiling is of congruent triangles with $6$ meeting at each vertex (a \emph{triangle tiling}), and where the resulting tiling is of alternating regular hexagons and triangles (the \emph{trihexagonal} tiling).

Based on our results, we outline the answers to the same set of basic questions, now considering tiling billiards. 

\vspace{1em}\noindent\emph{Are there periodic trajectories?}

We find periodic trajectories in every class of tiling billiards we study. For example, every triangle tiling has a periodic trajectory of period $6$ around a vertex (Corollary \ref{all-triangles}), and almost all have a period-$10$ trajectory around 2 vertices (Theorem \ref{per10}). We give examples of periodic trajectories in the trihexagonal tiling (Examples \ref{6theorem}, \ref{12periodic} and \ref{gooseheadtheorem}) with period $6,12$ and $24$, respectively.

\vspace{1em}\noindent\emph{Can trajectories escape?}  

We  study arrangements of finitely many lines in the plane, all of which trivially have escaping trajectories that reflect only once, and some of which have trajectories that escape by spiraling outward (Proposition \ref{periodic_unstable}).  

We also show that every right triangle tiling has an unbounded trajectory (Theorem \ref{bisecting}). In tiling billiards on tilings with translational symmetry, we frequently have a type of very regular escaping trajectory, which we call a \emph{drift-periodic} trajectory: it is not periodic, but is periodic up to translation, like a staircase or frieze pattern, escaping to infinity in two opposing directions. We show that some right triangle tilings have drift-periodic trajectories (Theorem \ref{dper}), and based on  computer experiments we conjecture that they all do (Conjecture \ref{everyrightdp}). 

We prove the existence of several families of drift-periodic trajectories on the trihexagonal tiling (Propositions \ref{drift120} and \ref{turning_drift}), and based on  computer experiments we conjecture that there are non-periodic escaping trajectories as well  (Conjecture \ref{nonperesc}). 

\vspace{1em}\noindent\emph{Can a trajectory fill a region densely? }

We show that the trihexagonal tiling has trajectories that fill regions of the plane with equally-spaced parallel lines that are arbitrarily close together (Theorem \ref{dense}), by construction. However, in the limit such a trajectory is periodic, not dense. Still, we conjecture (Conjecture \ref{denseexists}) that dense trajectories do exist. In contrast, triangle tilings appear not to have any dense trajectories (Conjecture \ref{trinodense}).

\vspace{1em}\noindent\emph{ Is the behavior stable under small perturbations of the starting point and direction of the trajectory, and under small perturbations of the billiard table?}

The most surprising observation that we made relates to this question: We found that the trihexagonal tiling exhibits very unstable behavior, in that a slight perturbation of the direction wildly changes the trajectory. For example, in trying to find periodic trajectories with our computer program, we had to be extremely precise with the mouse and a tiny change would lead to a trajectory that escaped to infinity.\footnote{A version of our program is available online at \url{http://awstlaur.github.io/negsnel/}.} In contrast, triangle tilings have extremely stable behavior; changing the direction, or starting location, or angles and lengths of the tiling triangle, only superficially changes the trajectory. For example, at one point we conjectured that some behavior does not depend on the starting position, because even a vast difference in moving the mouse seemed to have no effect on the trajectory. We had trouble proving it, and found that when we ``zoomed in'' many times in our program, the starting conditions did actually have a tiny effect. In this sense, triangle tilings are similar to outer billiards in their stability, and the trihexagonal tiling is similar to inner billiards in its instability. 

\vspace{1em}\noindent\emph{ Is the behavior invariant under affine transformations?}

The dynamics of the tiling billiards system is not invariant under affine transformations, because like the inner billiards system, it uses angles. For example, all triangle tilings are affinely equivalent, but different triangle tilings exhibit very different behavior: every right triangle tiling has an escaping trajectory (Theorem \ref{bisecting}), but the equilateral triangle tiling has only periodic trajectories of period $6$ around the vertices, and no escaping trajectories (Theorem \ref{mf}).\\

In this paper, we have results on periodic, drift-periodic and escaping trajectories. We have observations and conjectures about the existence of periodic trajectories on large classes of tilings, about the existence of dense trajectories, about  trajectories' stability, and about other basic questions about this dynamical system, but exploration of the system remains almost completely open.

\section{Results} 
\subsection{Previous results on tiling billiards}


Mascarenhas and Fluegel \cite{bloch} asserted that every trajectory in the  equilateral triangle tiling, and in the  $30^{\circ}$-$60^{\circ}$-$90^{\circ}$ triangle tiling, is periodic, and that every trajectory in the  square tiling is either periodic or drift-periodic. All of these tilings, and example trajectories, are in Figure \ref{regtilings}.

Engelman and Kimball \cite{2012} proved that there is always a periodic trajectory about the intersection of three lines (see Figure \ref{keresults}), and that there is no periodic orbit about the intersection of two non-perpendicular lines. In Section \ref{alex}, we generalize these results to certain divisions of the plane by any finite number of non-parallel lines.


Because neither \cite{bloch} nor \cite{2012} is in the published literature, we include these results and their proofs here, for completeness. We also add the regular hexagon tiling, so that we have all three tilings by regular polygons (even though the hexagon tiling is not two-colorable).

\begin{theorem}[Mascarenhas and Fluegel] \label{mf}
 Every trajectory is periodic in the equilateral triangle tiling and in the $30^{\circ}$-$60^{\circ}$-$90^{\circ}$ triangle tiling. Every trajectory in the square tiling and in the regular hexagon tiling is either periodic or drift-periodic.
\end{theorem}

\begin{figure} [!h]
\centering 
\includegraphics[height=80pt]{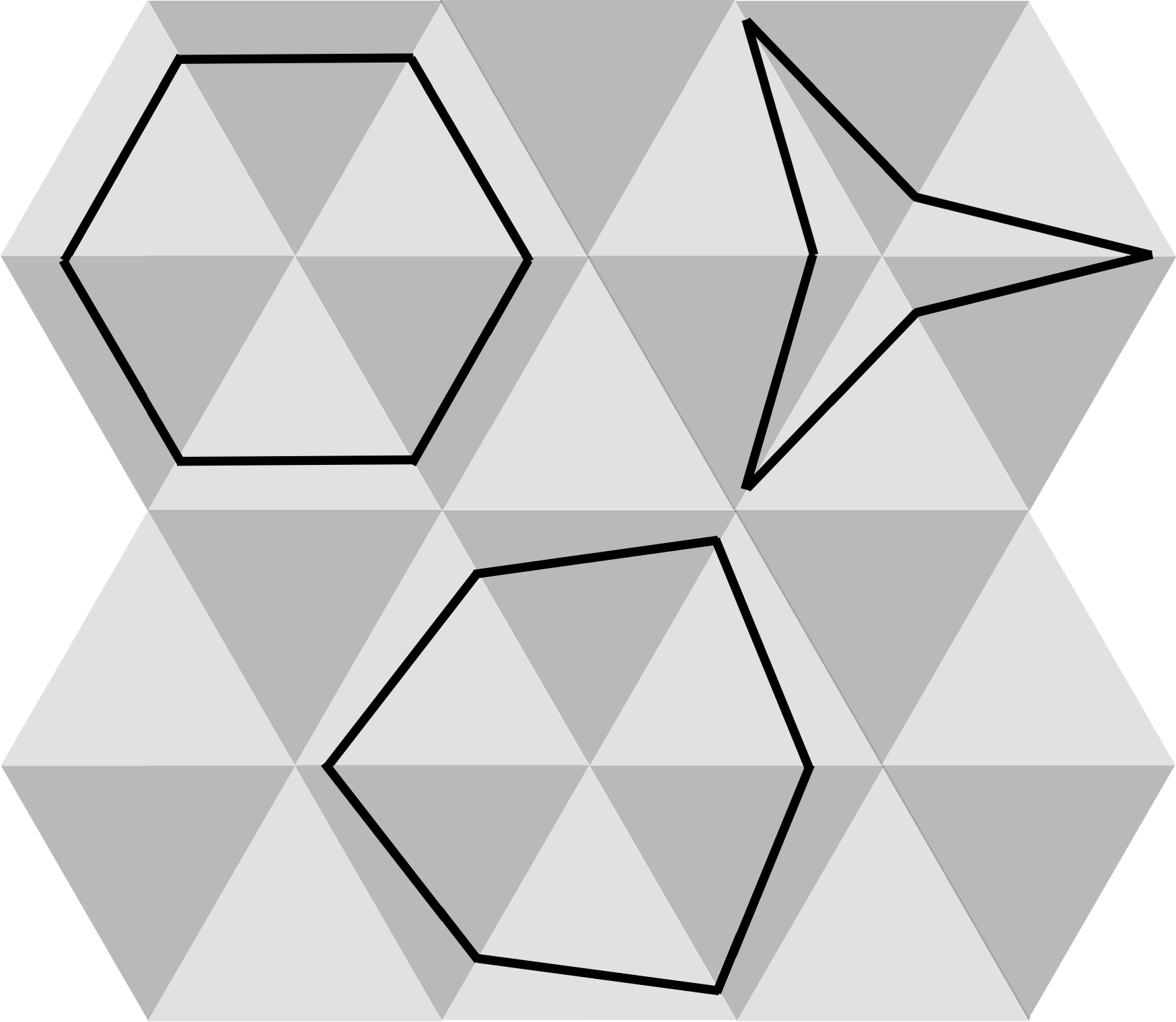} \ \ \ \
\includegraphics[height=80pt]{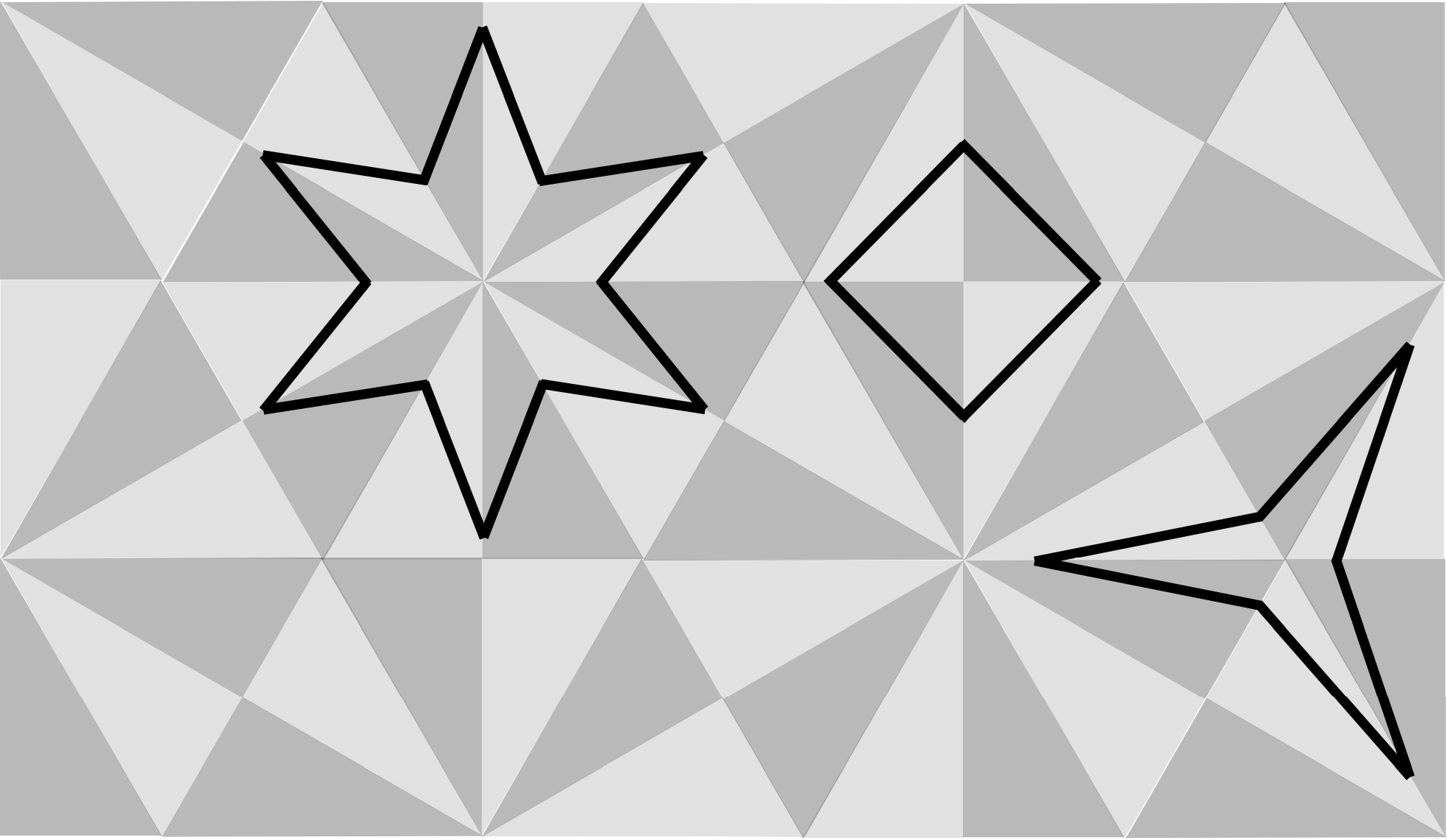} \\ 
\includegraphics[height=90pt]{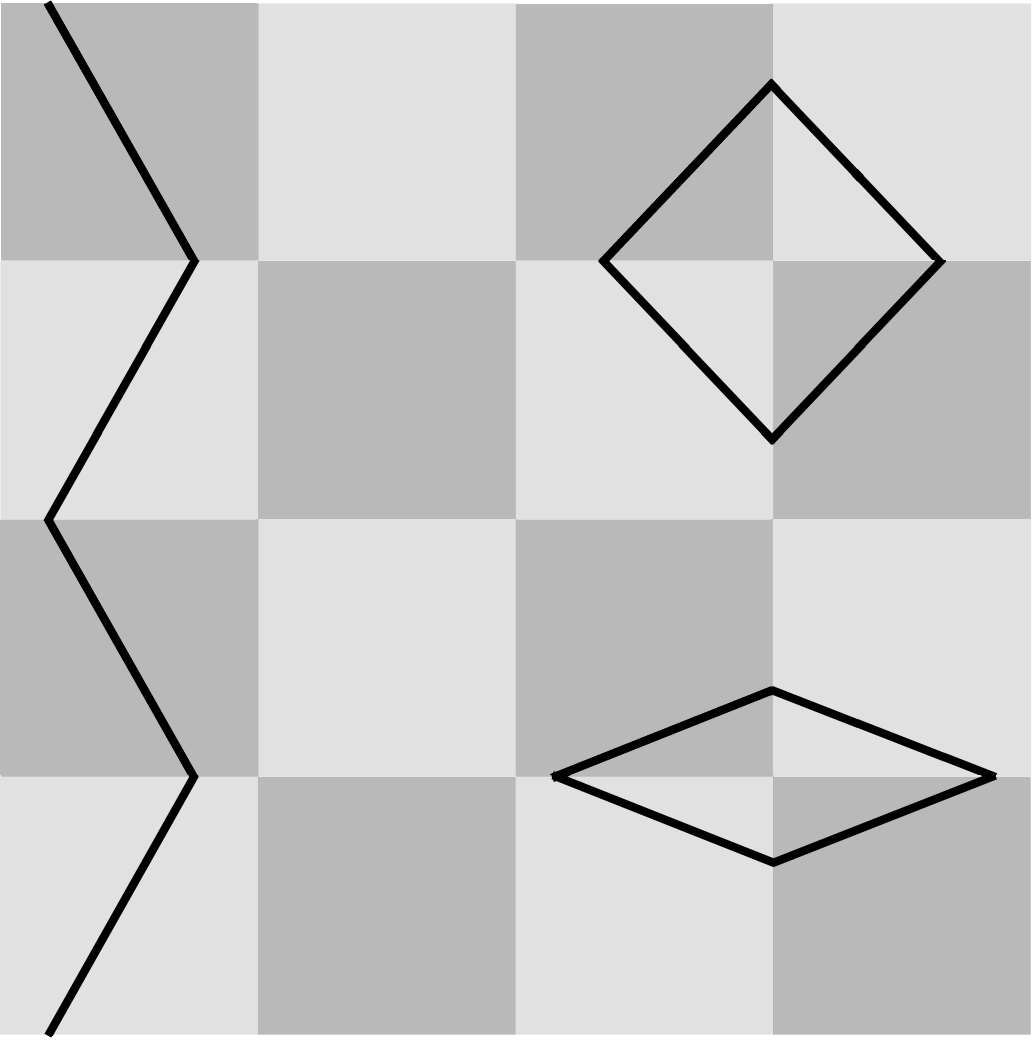} \ \  \ \ \ \ 
\includegraphics[height=90pt]{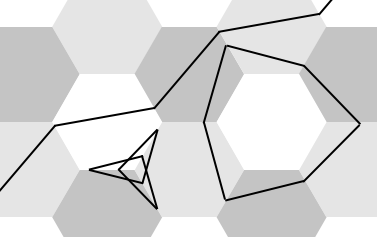}
\caption{Trajectories on the (a) equilateral triangle, (b) $30^{\circ}$-$60^{\circ}$-$90^{\circ}$ triangle, (c) square, and (d) regular hexagon tilings. Billiards in these four basic tilings have very simple dynamics}
\label{regtilings}
\end{figure}

\begin{proof}
Each of these tilings is reflection-symmetric across every edge of the tiling. (Note that there are many tilings of $30^{\circ}$-$60^{\circ}$-$90^{\circ}$ triangles; Mascarenhas and Fluegel chose the one pictured in Figure \ref{regtilings}b.) Thus, every trajectory that crosses a given edge is also reflection-symmetric across that edge. In the triangle tilings, a trajectory going from one edge of a triangle to another always crosses an adjacent edge, so  under the reflectional symmetry it closes up into a periodic path with the valence of the vertex between the adjacent edges, which is always $6$ in the equilateral triangle tiling, and is $4$, $6$ or $12$ in the $30^{\circ}$-$60^{\circ}$-$90^{\circ}$ triangle tiling tiling (Figure \ref{regtilings} (a) and (b)).

In the square tiling, a trajectory that starts on one edge either crosses the adjacent edge or the opposite edge. If it crosses the adjacent edge, it circles around the vertex between those edges, so it closes up into a periodic path with the valence of the vertex: $4$. If it crosses the opposite edge, it ``zig-zags'' in a drift-periodic path of period 2 (Figure \ref{regtilings} (c)). Similarly, in the regular hexagon tiling, a trajectory crosses the adjacent edge, the next-to-adjacent edge, or the opposite edge. If it crosses the adjacent edge or next-to-adjacent edge, it closes up into a periodic path of period $6$, and if it crosses the opposite edge, it zig-zags as in the square tiling (Figure \ref{regtilings} (d)).
\end{proof}

There is an alternative proof of Theorem \ref{mf} using ``folding.'' In inner billiards, a common technique is to ``unfold'' the trajectory across the edges of the table, creating a new copy of the table in which the trajectory goes straight. Analogously but going the opposite direction, in tiling billiards we can ``fold up'' the tiling so that congruent tiles are on top of each other, and the trajectory goes back and forth as a line segment between edges of a single tile (Figure \ref{folding}).

\begin{figure}[!h] 
\centering 
\includegraphics[height=100pt]{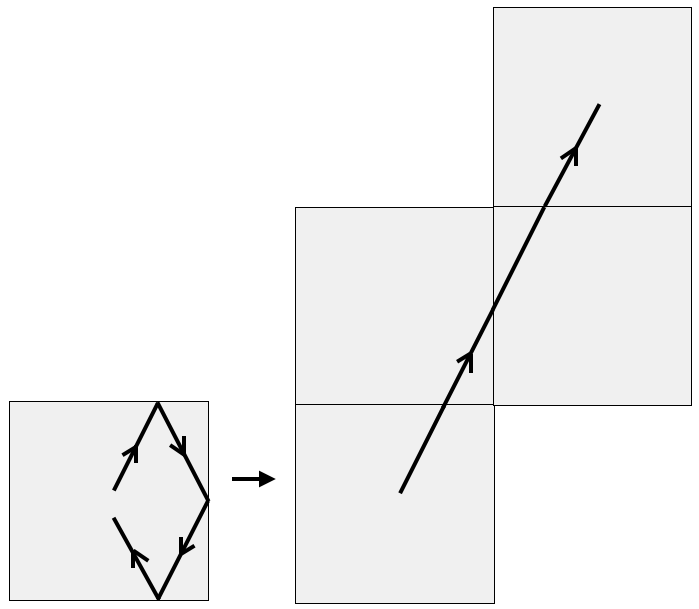}  \ \ \ \ \ \ \ \ \ \ \  \ \ \ \ \ \ 
\includegraphics[height=67pt]{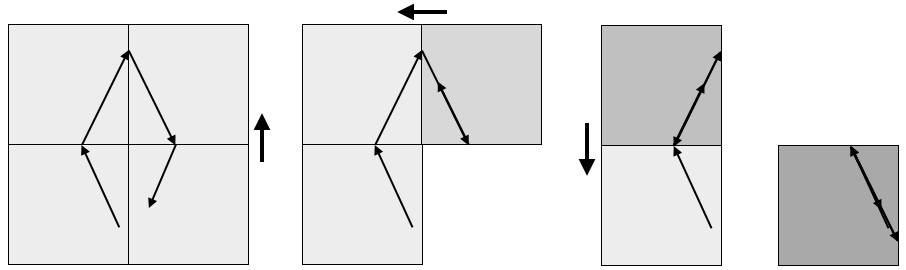} 
\caption{(a) Unfolding a trajectory on the square billiard table into a straight line (b) Folding a trajectory on the square tiling into a segment}
\label{folding}
\end{figure}

\begin{proof}[Alternative proof of Theorem \ref{mf} using folding]
Each of the four tilings is reflection-symmetric across every edge of the tiling, so when a trajectory crosses an edge from one tile to another, that edge is a line of symmetry for the two adjacent tiles, and we fold one tile onto the other. When we fold up a trajectory this way, it is a line segment going back and forth between two edges of the tile, so every trajectory is periodic or drift-periodic. For a periodic trajectory, the period is the number of copies needed for the trajectory to return to the same edge in the same place, which for a vertex of valence $v$ is $v$ when $v$ is even and $2v$ when $v$ is odd. For a drift-periodic trajectory, the period is the number of copies needed for the trajectory to return to a corresponding edge in the same place, which in both the square and hexagon tiling is $2$.
\end{proof}

We will use the same technique to prove the Trajectory Turner Lemma \ref{turner}.

\begin{theorem}[Engelman and Kimball] \label{ek}
(a) There is a periodic trajectory about the intersection of two lines if and only if the lines are perpendicular. (b) There is always a periodic trajectory about the intersection of three lines.
\end{theorem}

\begin{figure}[!h] 
\centering 
\includegraphics[width=0.4\linewidth]{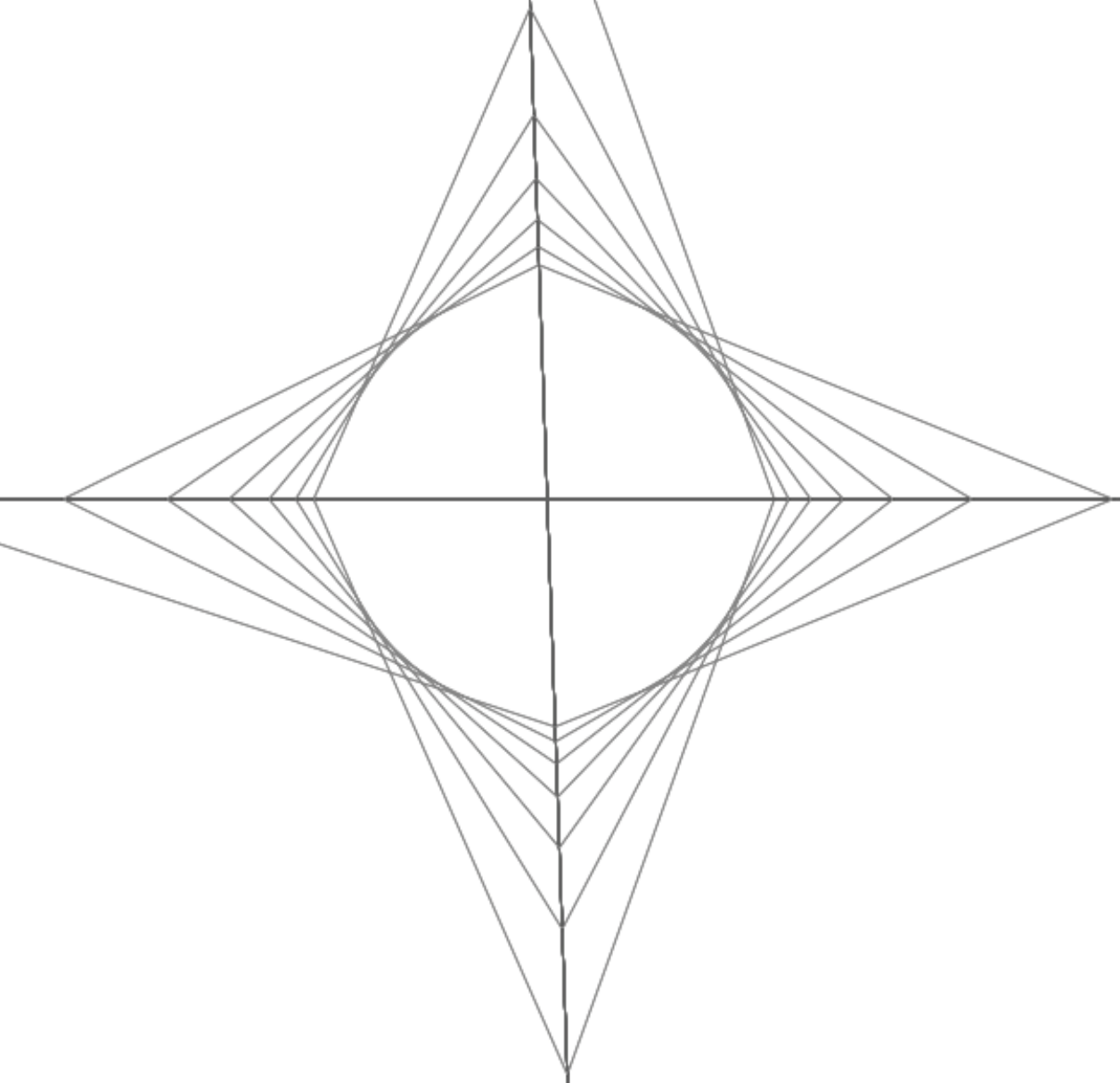} \ \ \ \ \ \ \
\includegraphics[width=0.4\linewidth]{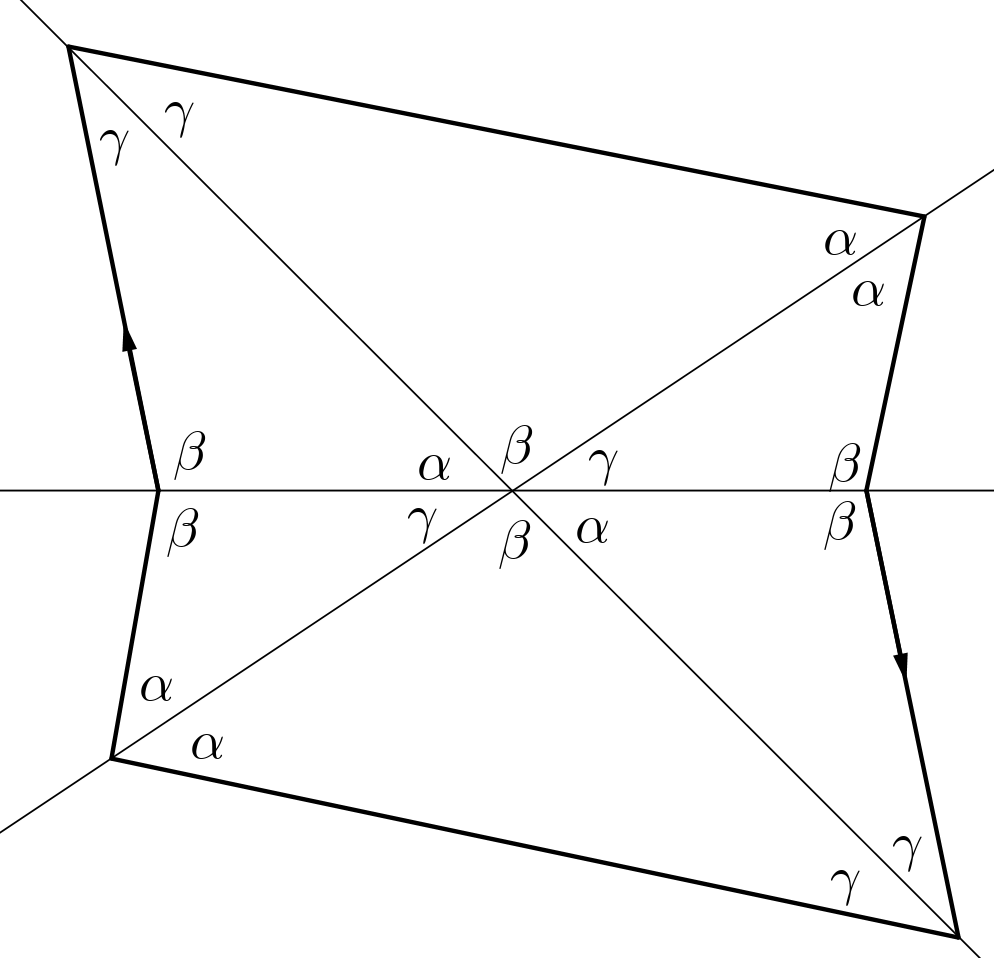}
\caption{(a) There is no periodic trajectory about two non-perpendicular lines (b) There is always a periodic trajectory about the intersection of three lines}
\label{keresults}
\end{figure}

\begin{proof}
(a) Suppose that the two lines intersect at an angle $\alpha$, and the initial angle between the trajectory and a line is $\theta$, measured on the same side as $\alpha$. If a trajectory is periodic, then it must be $4k$-periodic for some natural number $k$, since it must cross both lines twice to return to its starting trajectory. If the initial angle is $\theta$, the first four angles between the trajectory and the lines are $\theta, \alpha-\theta,\pi+\theta-2\alpha, 3\alpha-\pi-\theta$, and then $\theta+4\alpha-2\pi$ with the original line. By induction, we see that the $k^{\text{th}}$ time the trajectory returns to the original line, it will make the angle $\theta + k(4\alpha - 2\pi)$. For the trajectory to be periodic, we need this to be equal to the initial angle, so in a $4k$-periodic orbit, $\theta+k(4\alpha-2\pi)=\theta$, which occurs if and only if $\alpha = \pi/2$. Note that the intersection angle between the trajectory and the lines is measured modulo $2\pi$, but once this angle is greater than both $\alpha$ and $\pi - \alpha$, the trajectory escapes, so we do not consider $\theta > 2\pi$.

When $\alpha = \pi/2$, the lengths are then trivially equal, so the trajectory is periodic, and must have period 4. 
An example of the non-perpendicular behavior is given in Figure \ref{keresults} (a), where the lines meet at $88^\circ$. 

(b) Label the angles of intersection between the three lines $\alpha, \beta$ and $\gamma$, so $\alpha+\beta+\gamma=\pi$. If we start a trajectory on the line between angles $\alpha$ and $\gamma$ with initial angle $\beta$ (Figure \ref{keresults} (b)), then when the trajectory returns to the same line, by symmetry it will have the same angle. It will also be at the same point; the length is calculated via a simple Law of Sines argument. Then the trajectory is periodic with period $6$.
\end{proof}

We will use an analogous construction to show that there is a periodic trajectory about any odd number of non-parallel lines, in Theorem \ref{yaythm}.

Glendinning \cite{glendinning} studies a system that is similar to ours: a ``chessboard'' tiling where the two colors of tiles have different positive indices of refraction, in particular where the refraction coefficient is greater than $1$. This is similar to our setup, except that the angle at a boundary changes, which leads to different behavior than in our system. He finds that the number of angles with which any ray intersects the square lattice is bounded, and that if the refraction coefficient is high enough, one can use interval exchange transformations to describe the dynamics.

\subsection{Results}
In this paper, we extend the results presented in \cite{2012} and \cite{bloch}, and we provide several novel results in tiling billiards. We give a summary below of our results for line arrangements, triangle tilings, and the trihexagonal tiling. 


The first set of results concerns a division of the plane by a finite number of non-parallel lines (a \emph{line arrangement}). We search for periodic trajectories, and find that when the number of lines is odd, the condition is simple:

{\bf Theorem \ref{yaythm}:} For any arrangement of a finite number $n$ of non-parallel lines with $n \geq3$ odd, there is always a periodic trajectory.

When the number of lines is even, the situation is a bit more delicate:

{\bf Theorem \ref{superyaythm}:} For $n \geq2$ even, there is a periodic trajectory if and only if the counter-clockwise angles $\alpha_0,\ldots,\alpha_{n-1}$, between consecutive lines ordered by angle, satisfy
$$0 = \alpha_{0} - \alpha_{1} +\alpha_2-\alpha_3+\dots+ \alpha_{n-4} - \alpha_{n - 3} +\alpha_{n-2} - \alpha_{n-1}.$$


\vspace{1em}

The second set of results concerns \emph{triangle tilings}: tilings by congruent triangles meeting $6$ at a vertex formed by dividing the plane by lines, or equivalently a grid of parallelograms with parallel diagonals. We find some nice results, particularly concerning periodic trajectories. We also find what we call \emph{drift-periodic} trajectories, which are like a staircase or frieze pattern: there exist infinitely many distinct congruent tiles, where the trajectory crosses corresponding edges in the same place, at the same angle, and consecutive corresponding points differ by a constant vector.  

{\bf Theorem \ref{iso}:} In an edge-to-edge tiling of the plane by congruent isosceles triangles meeting $6$ at a vertex formed by dividing the plane by lines, all trajectories are either periodic or drift-periodic.

The above Theorem does not guarantee that an isosceles triangle tiling \emph{has} an escaping trajectory; for example, the equilateral triangle tiling (Figure \ref{regtilings}) does not. By contrast, every right triangle tiling does:

{\bf Theorem \ref{bisecting}:} Every edge-to-edge tiling of the plane by congruent right triangles with parallel diagonals has an escaping trajectory. 

We prove this theorem by construction: if a trajectory bisects a hypotenuse, then it is unbounded. 
For certain  right triangles, we can guarantee not only that there is an escaping orbit, but also that the escaping orbit is drift-periodic:

{\bf Theorem \ref{dper}:} In an edge-to-edge tiling of the plane by congruent right triangles with parallel diagonals, if the smallest angle of a right triangle tiling is $\frac{\pi}{2n}$ for some $n \in \mathbb{N}$, then a drift-periodic orbit exists.

By Theorem \ref{ek}, every triangle tiling has a periodic trajectory of period $6$ that circles one vertex. It turns out that almost every triangle tiling also has a periodic trajectory of period $10$ that circles two vertices:

{\bf Theorem \ref{per10}:} Every edge-to-edge tiling of the plane by congruent triangles meeting $6$ at a vertex formed by dividing the plane by lines, except tilings by isosceles triangles with vertex angle greater than or equal to $\frac{\pi}{3}$, has a periodic orbit with period $10$.

%
%

\vspace{1em}

The third set of results concerns the complex behavior of the \emph{trihexagonal tiling}, which has a regular hexagon and an equilateral triangle meeting at every edge. We first prove several results about local behavior of trajectories in this tiling (Lemmas \ref{turner}, \ref{quadrilateral}, \ref{quad_triangle} and \ref{pent_lemma}). Then we use the local results to prove the existence of several periodic trajectories in the tiling (Examples \ref{6theorem}, \ref{12periodic} and \ref{gooseheadtheorem}), and then to find infinite families of related drift-periodic trajectories (Propositions \ref{drift120} and \ref{turning_drift}). We show that the segments of such trajectories can be very closely packed in the regions they visit:

{\bf Theorem \ref{dense}:}
The trihexagonal tiling exhibits trajectories that fill infinite regions of the plane with line segments that are arbitrarily close together.

\vspace{1em}

We now begin our detailed study of several types of tiling billiards systems.

\section{Line arrangements} \label{alex}

In this paper, we consider the case of tiling billiards where the tiling is created by dividing up the plane by lines. The simplest such system is one with a finite number of non-parallel lines, which we call a \emph{line arrangement} and study in this section.

 Gr{\"u}nbaum also studied line arrangements in \cite{grunbaum}. 


\begin{definition}\label{polyalpha}
Identify one of the lines, and name it $l_0$. Label the remaining lines $l_1, \ldots l_{n-1}$ in order of increasing counter-clockwise angle from $l_0$.
Define the angle $\alpha_{i}$  to be the counter-clockwise angle between lines $l_{i}$ and $l_{i+1}$, considering the indices modulo $n$. 
\end{definition}

\begin{lemma}
\label{alpha_lemma} If we label a finite number of lines $l_0,\ldots l_{n-1}$ counter-clockwise, and let $\alpha_i$ be the counter-clockwise angle between lines $l_{i}$ and $l_{i+1}$ (modulo $n$), then
  $\alpha_{0} + \alpha_{1} + \dots + \alpha_{n-1} = \pi$.
\end{lemma}
\begin{proof}
Parallel-translate the lines so that they coincide at a single point. Then the result is clear.
\end{proof}

\begin{definition}
The convex hull of the intersection points of the lines is called the \emph{central zone}. Because we are considering a finite number of non-parallel lines, the number of intersection points is also finite, so the central zone is bounded (Figure \ref{polygon}). 

For a trajectory $\tau$, let $\theta_i$ be the angle at the trajectory's $i^{\text{th}}$ intersection with a line, measured on the side of the central zone. (This line will be $l_k$, where \mbox{$k \equiv i$ (mod $n$)}.) For convenience, we define the functions $\theta_i(\tau)$, which return this angle for a given trajectory $\tau$.

In any division of the plane containing regions of infinite size, we say that a trajectory \emph{escapes} if it eventually stops refracting across lines and remains in a single region.

If a trajectory  neither enters the central zone, nor escapes, for at least $2n$ iterations (refractions), we call the trajectory \emph{good}.
\end{definition}

We restrict our analysis in this section to good trajectories. This is a reasonable choice because the requirement to stay outside the central zone isn't taxing: one can  ``zoom out" until the central zone is of arbitrarily small size. 
We choose to avoid the central zone because inside the central zone, we cannot guarantee that the trajectory crosses the lines in order, but outside, we can:

\begin{lemma} \label{inorder}
Consider a counter-clockwise good trajectory on a line arrangement of $n$ non-parallel lines, i.e. a trajectory that neither enters the convex hull of the lines' intersection points nor escapes to infinity for at least $2n$ refractions. Such a trajectory crosses the lines in numerical order: $l_{0}, l_{1}$, $\dots, l_{n-1}, l_{0}, l_{1}$, etc. 
\end{lemma}

 Note that a clockwise trajectory would cross the lines in reverse numerical order; without loss of generality, we may assume that the trajectory travels counter-clockwise.

 \begin{proof}
  Assume the trajectory has just crossed $l_{i}$ from the side of $l_{i-1}$. There are four options (see Figure \ref{polygon}): to enter the central zone, to escape, to re-cross $l_{i}$, or to cross $l_{i+1}$. The first two options are impossible by assumption, and the third option is impossible because two lines can only intersect once in the plane. Thus, the trajectory must cross $l_{i+1}$.
 \end{proof}

\begin{figure}[h!]
\centering
\includegraphics[width=250pt]{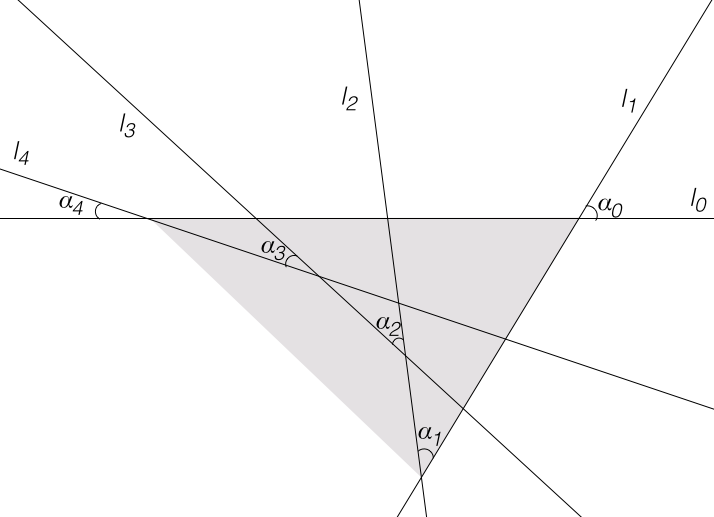}
\caption{A line arrangement with $5$ lines. The central zone is shaded. }
\label{polygon}
\end{figure}

We will use the following classical results about planar geometry.

\begin{lemma}\label{classify}
\begin{enumerate}
\item A  translation has a $1$-parameter family of parallel fixed lines. \label{transfam}
\item A composition of reflections across an odd number of lines is a reflection or glide-reflection. \label{oddrefl}
\item If the counter-clockwise angle between intersecting lines $\ell_i$ and $\ell_{i+1}$ is $\alpha_i$ for $i=0,2,\ldots,2n-2$, then the composition of reflection through $\ell_0, \ell_1,\ldots,\ell_{n-1}$ is a counter-clockwise rotation with angle $2(\alpha_0+\alpha_2+\ldots+\alpha_{n-2})$. \label{ccwangle}
\end{enumerate}
\end{lemma}

For proofs of these results, see \cite{barker}.

\begin{theorem}\label{yaythm}
For a line arrangement of $n$ non-parallel lines with $n\geq 3$ odd, there is always a trajectory of period $2n$. A periodic direction is given by the initial angle (measured between the trajectory and $l_0$, on the same side as the central zone) 
\begin{equation}\label{odd-condition}
\theta=\alpha_1+\alpha_3+\ldots+\alpha_{n-2}.
\end{equation} 
\end{theorem}

\begin{proof}
Let $T$ be the composition of reflections in $l_0, \ldots, \l_{n-1}$. Any good trajectory (see below for its existence) will circle the central zone before it gets back to $l_0$, reflecting across the $n$ lines in order twice, so we consider the transformation $T^2$. Since $n$ is odd, by Lemma \ref{classify} part \ref{oddrefl}, $T$ is either a reflection or a glide-reflection, so $T^2$ is either the identity or a translation, respectively. If $T^2$ is the identity, then there is a two-parameter family of periodic trajectories. If $T^2$ is a translation, by Lemma \ref{classify} part \ref{transfam} there is a $1$-parameter family of parallel fixed lines. Choose any fixed line that does not intersect the central zone; this provides a periodic trajectory.  Such a trajectory circles the central zone once, refracting across each line twice, so it has period $2n$.

We will now show that there exists a good trajectory with initial angle $\theta$ given by (\ref{odd-condition}). By construction, for $0\leq i \leq n-1$,
$$\theta_i = \ldots +  \alpha_{i-4}  +  \alpha_{i-2}  +  \alpha_{i+1}  +  \alpha_{i+3}  + \ldots,$$
where the angles included in the sum are the $\alpha_k$ with  $k$ satisfying $0\leq k \leq n-1$, in other words all those that are defined. We also have $\theta_i = \theta_{i\text{\  (mod \ }n)}$ for $i \geq n$. 

From this we can see that $0 < \theta_i < \pi$, the first inequality because no two lines are parallel, and the second inequality because the sum of the $\alpha_i$ is $\pi$ (Lemma \ref{alpha_lemma}). Since $\theta_i > 0$, it is possible to start our trajectory far enough from the central zone that it does not enter the central zone. Since $\theta_i < \pi$, the trajectory does not escape: the piece of trajectory between the points of intersection where $\theta_i$ and $\theta_{i+1}$ are measured forms a (non-degenerate) triangle with $l_i$ and $l_{i+1}$, which has angles $\alpha_i, \theta_i, \theta_{i+1}$.
\end{proof}

\begin{corollary} \label{ngoncorollary}
Let $n\geq 3$ be odd. If all of the angles are equal, i.e. $a_i=\pi/n$ for each $i$, then the angle condition (\ref{odd-condition}) reduces to $\theta_0=\frac{n-1}{2n}\pi$, which gives a periodic direction; the trajectory forms an equiangular polygon. 
\end{corollary}

Corollary \ref{ngoncorollary} applies in the case that the lines are the extensions of the edges of a regular polygon, or any equiangular polygon.

%

\begin{proposition} \label{notrot}
For a line arrangement of $n$ non-parallel lines $l_0, \ldots, l_{n-1}$ with $n\geq 2$ even, ordered by increasing counter-clockwise angle, where $\alpha_i$ is the counter-clockwise angle between lines $l_{i}$ and $l_{i+1}$ (modulo $n$), let $T$ be the composition of reflections through $l_0, \ldots, l_{n-1}$, and consider the following angle condition:
\begin{equation} \label{even-condition}
0 = \alpha_{0} - \alpha_{1} +\alpha_2-\alpha_3+\dots+ \alpha_{n-4} - \alpha_{n - 3} +\alpha_{n-2} - \alpha_{n-1}.
\end{equation}
$T$ is a rotation by $\pi$, and $T^2$ is the identity, if and only if (\ref{even-condition}) is satisfied. 
\end{proposition}

\begin{proof}
By Lemma \ref{alpha_lemma}, the sum of all the $\alpha_i$ is $\pi$, so an equivalent statement to (\ref{even-condition}) is $$\alpha_0 + \alpha_2 + \ldots + \alpha_{n-2}=\alpha_1 + \alpha_3 + \ldots + \alpha_{n-1}=\pi/2.$$
By Lemma \ref{classify} part \ref{ccwangle}, a composition of reflections through $l_0, l_1, \ldots, l_{n-1}$ is a rotation by angle $2(\alpha_0 + \alpha_2 + \ldots + \alpha_{n-2})$, so $T$ is a rotation by angle $2(\pi/2)=\pi$ if and only if (\ref{even-condition}) is satisfied.  Any good trajectory will circle the central zone before it gets back to $l_0$, reflecting across the $n$ lines in order twice, so we should study the transformation $T^2$. $T^2$ is a rotation of angle $2\pi$, and is thus the identity, if and only if (\ref{even-condition}) is satisfied, as desired.
\end{proof}

\begin{corollary} \label{for-spiral} For a trajectory $\tau$ on a line arrangement defined as in Proposition \ref{notrot}, let $\theta_i(\tau)$ be the angle between $\tau$ and $\l_i$, measured on the side of the central zone.

If $\theta_n(\tau) = \theta_0(\tau)$, then $\theta_{i+n}(\tau)=\theta_i$ for all $i\geq 0$.

If $\theta_n(\tau) = \theta_0(\tau)$, then $\tau$ is periodic.
\end{corollary}

The result of Corollary \ref{for-spiral} is surprising: We need only check an angle condition to see that a given trajectory is periodic; we need not check that the trajectory returns to the same point on $l_0$.

\begin{theorem}\label{superyaythm}
For a line arrangement of $n$ non-parallel lines $l_0, \ldots, l_{n-1}$ with $n\geq 2$ even, ordered by increasing counter-clockwise angle, let $\alpha_i$ be the counter-clockwise angle between lines $l_{i}$ and $l_{i+1}$ (modulo $n$). If (\ref{even-condition}) is satisfied, then every good trajectory is periodic. If (\ref{even-condition}) is not satisfied, then no good trajectory is periodic.
\end{theorem}

\begin{proof}
If (\ref{even-condition}) is satisfied, then by Proposition \ref{notrot}, $T^2$ is the identity, so it has a two-parameter family of fixed lines; as in Theorem \ref{yaythm}, this yields a periodic trajectory of period $2n$.

Suppose that (\ref{even-condition}) is not satisfied. For a good trajectory $\tau$, $\theta_{2n}(\tau) = \theta_0(\tau) + 2(\alpha_0-\alpha_1+\ldots-\alpha_{n-1})$. This means that twice the amount that the number in the right-hand side of (\ref{even-condition}) differs from $0$ is added to  $\theta_0$ each time $\tau$ circles the central zone. Eventually the angle $\theta_{2kn}(\tau)$ will be less than $0$ (if this number is negative) or greater than $\pi$ (if this number is positive), so the trajectory will enter the central zone or escape, so no such $\tau$ is periodic.
%
\end{proof}

We also consider the special case where all of the lines coincide, so the central zone is just the point of coincidence.

\begin{corollary} \label{n-lines}
Divide the plane by $n\geq 2$ lines $l_0, \ldots, l_{n-1}$ coinciding at a point, ordered by increasing counter-clockwise angle, where $\alpha_i$ is the counter-clockwise angle between lines $l_{i}$ and $l_{i+1}$ (modulo $n$), and consider a trajectory $\tau$ with initial angle $\theta$.
\begin{enumerate}
\item If $n$ is odd and (\ref{odd-condition}) is satisfied, then $\tau$ is periodic.
\item If $n$ is even and (\ref{even-condition}) is satisfied, then every non-escaping trajectory is periodic; if (\ref{even-condition}) is not satisfied, then every trajectory escapes.
\end{enumerate}
\end{corollary}

These follow directly from Theorem \ref{yaythm} and Theorem \ref{superyaythm}.

The above results show when periodic trajectories exist, and how to construct them. Now we will show that some periodic trajectories are stable and some are not. We will also show that when a trajectory is not periodic, it can spiral. 

\begin{definition}
\label{stable}
A periodic trajectory is {\it stable} if, when the initial angle is changed by some arbitrarily small angle, the trajectory hits the same series of edges and then returns to its starting angle and location. 
A non-periodic trajectory starting on $l_0$ \emph{spirals} outward if it circles the central zone and returns to $l_0$ with the same angle but further from the central zone. 
\end{definition}

\begin{proposition}
\label{periodic_unstable}


Consider a line arrangement of $n$ non-parallel lines $l_0, \ldots, l_{n-1}$ ordered by increasing counter-clockwise angle. Let $\tau$ be a good trajectory starting on $l_0$ going counter-clockwise, and let $\theta_i(\tau)$ be the angle between $\tau$ and $\l_i$, measured on the side of the central zone, and suppose that $\theta_0(\tau) = \theta_n(\tau)$. Then $\tau$ is periodic. If $n$ is even, $\tau$ is stable. If $n$ is odd, $\tau$ is not stable; the perturbed trajectory spirals. 
\end{proposition}
\begin{proof}
Since $\theta_{n} = \theta_{0}$, the trajectory is periodic by Corollary \ref{for-spiral}, so it suffices to examine the stability.

Let  $\tau'$ be the trajectory that starts in the same place as $\tau$, and has initial angle $\theta_0(\tau') = \theta_0(\tau)+\epsilon$, where  $\epsilon$ is small enough (positive or negative) so that $\tau'$ is still a good trajectory.

Then we have 
\begin{align*}
\theta_{0}(\tau') &= \theta_{0}(\tau) + \epsilon \\
\theta_{1}(\tau') &= \theta_{1}(\tau) - \epsilon \\
\vdots \\
\theta_{i}(\tau') &= \theta_{i}(\tau) +(-1)^i \epsilon.
\end{align*}


If $n$ is even, then $\theta_{n}(\tau') = \theta_{n}(\tau) + \epsilon = \theta_{0}(\tau')$. Since the perturbation $\tau'$ of $\tau$ hits the same series of edges and returns to its starting angle and location, $\tau$ is stable.

If $n$ is odd, then $\theta_{n}(\tau') = \theta_{n}(\tau) - \epsilon = \theta_{0}(\tau') - 2\epsilon$. 
By Corollary \ref{n-lines}, $\tau$ is periodic, so we still have $\theta_{2n} = \theta_{0}$, and $\tau'$ spirals.
\end{proof}

\section{Triangle Tilings} \label{jenny}

In this section, we investigate tiling billiards where lines divide the plane in a very regular way, into congruent triangles. The behavior of the system in some cases is simple, perhaps because of the substantial symmetry.

\begin{definition} \label{triangletiling}
A {\it triangle tiling} is a covering of the Euclidean plane with non-overlapping congruent copies of the \emph{tiling triangle} so that the tiling is a grid of parallelograms with parallel diagonals.
\end{definition}

The valence of every vertex of a triangle tiling is $6$. Note that the reflection-symmetric $30^{\circ}$-$60^{\circ}$-$90^{\circ}$ triangle tiling in Figure \ref{regtilings} is \emph{not} a triangle tiling by our definition. Since every vertex of a triangle tiling is the intersection of three lines, we have the following result:

\begin{corollary}[to Theorem \ref{ek}]\label{all-triangles}
Every triangle tiling has a trajectory of period $6$ around each vertex.
\end{corollary}

We have several elementary observations about triangle tilings that lead to surprisingly powerful results, the first of which is the following Lemma:

\begin{lemma}[Angle Adding Lemma]\label{angle-adding}
Consider a trajectory that consecutively meets the two legs of the tiling triangle that form angle $\alpha$. If the angle that the trajectory makes with the first leg, on the side away from $\alpha$, is $\theta$, then the angle that the trajectory makes with the second leg, on the side away from $\alpha$, is $\theta + \alpha$.
\end{lemma}

\begin{proof}
If a trajectory hits the legs of angle $\alpha$, then it forms a triangle, where one angle is $\alpha$ and another is the initial angle $\theta$ of the trajectory. So the exterior angle of the third angle must be $\theta + \alpha$, as shown in Figure \ref{angleadds}.
\end{proof}

\begin{figure}[h!]
\centering
\includegraphics[scale=.8]{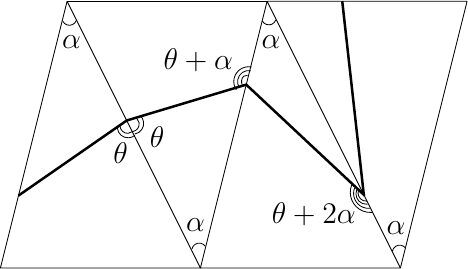}
\caption{The angle a trajectory makes with a leg of a tiling triangle increases by $\alpha$ on the side of the trajectory away from the angle $\alpha$.}
\label{angleadds}
\end{figure}


\subsection{Isosceles triangle tilings}

\begin{theorem} \label{classification}
Consider an isosceles triangle tiling, and let the vertex angle be $\alpha$.
\begin{enumerate}
\item All trajectories are either periodic or drift-periodic.
\item Consider a trajectory making angle $\theta < \alpha$ with one of the congruent legs. Let $n \in \mathbb{N}$ be the unique value such that $\pi - \alpha \leq \theta + n\alpha < \pi$. If n is even, the maximum period of a drift-periodic trajectory is $2n+4$ and the maximum period of a periodic trajectory is $2n + 2$; if $n$ is odd, the maximum period of a drift-periodic trajectory is $2n + 2$ and the maximum period of a periodic trajectory is $2n + 4$.
\end{enumerate}
\label{iso}
\end{theorem}

\begin{proof}
\begin{enumerate}
\item Call each line containing the bases of the isosceles triangles the {\it base line}. Note that an isosceles triangle tiling is reflection-symmetric across the base lines, and therefore so are the trajectories. So if a trajectory crosses the same base line in two places, then it must make a loop upon reflection across that line; the trajectory is periodic (Figure \ref{isosceles}a). If a trajectory crosses two distinct base lines, then upon reflection across either of these lines, the trajectory will hit yet another base line, in the same place and at the same angle on the side of a triangle corresponding to the previous place the trajectory met a base line (Figure \ref{isosceles}b); the trajectory is drift-periodic.

\begin{figure}[h!]
\centering
\includegraphics{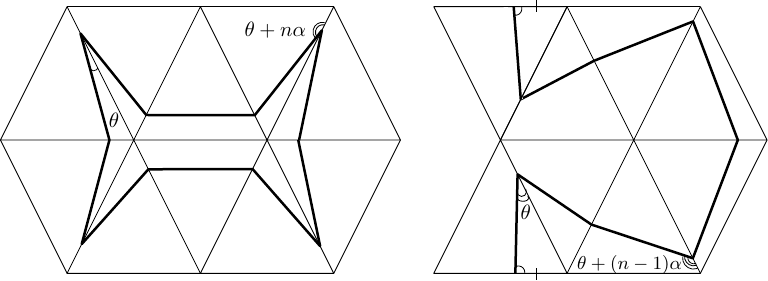}
\caption{Two trajectories in an isosceles triangle tiling: (a) a periodic trajectory and (b) a drift-periodic trajectory. Note that these trajectories are reflection-symmetric across the base line.}
\label{isosceles}
\end{figure}

Furthermore, every trajectory falls into one of these two cases because a trajectory cannot travel indefinitely without crossing a base: each time a trajectory goes from leg to leg in the tiling, by the Angle Adding Lemma $\alpha$ is added to the angle the trajectory makes with the next leg. Eventually, this angle is greater than or equal to $\pi - \alpha$ and the trajectory meets a base. So every trajectory is either periodic or drift-periodic.
\item Note that by the Angle-Adding Lemma, $\theta$ is the smallest angle that the trajectory makes with the edges. By the definition of $n$, $\theta + n\alpha$ is the largest angle a trajectory can make with a leg before the trajectory meets a base line. Suppose that a trajectory meets this maximum number of legs $n$. If $n$ is even, then the trajectory meets another base line after making the angle $\theta + n\alpha$ with a leg and, by the reflection symmetry of the tiling across the base lines, is drift-periodic. By the Angle Adding Lemma and reflection symmetry, a drift-periodic trajectory must have $2n$ points where the trajectory makes the angles $\theta + \alpha, \theta + 2\alpha, \dots, \theta + n\alpha$. Add to this the 4 points when the trajectory is traveling to and from a base edge, and the maximum period of a drift-periodic trajectory is $2n+4$. Similarly, the maximum period of a periodic trajectory is $2(n-1) + 4 = 2n + 2$. If $n$ is odd, then the trajectory meets a base line a second time after making the angle $\theta + n\alpha$ with a leg and is periodic; the maximum period of a drift-periodic trajectory is $2(n-1) + 4 = 2n + 2$ and the maximum period of a periodic trajectory is $2n + 4$.

\end{enumerate}
\end{proof}

Theorem \ref{classification} shows that every trajectory in an isosceles triangle tiling is either periodic or drift-periodic, and we might wonder if all isosceles triangles yield both. No: we know (Corollary \ref{all-triangles}) that every triangle tiling has a periodic trajectory, but Theorem \ref{mf} shows that the equilateral triangle tiling (a special case of an isosceles triangle) does \emph{not} have a drift-periodic trajectory. We conjecture that this is not the only exception:

\begin{conjecture}
An isosceles triangle tiling has an escaping trajectory if and only if its vertex angle is \emph{not}  of the form $\pi/(2n+1)$ for $n\geq 1$.
\end{conjecture}

\subsection{Right triangle tilings}

For us, a \emph{right triangle tiling} is an edge-to-edge tiling of the Euclidean plane with congruent right triangles with axis-parallel perpendicular edges, such that the hypotenuses are the negative diagonals of the rectangles formed by the perpendicular edges. We refer to a tiling triangle that lies below its hypotenuse as a {\it lower tiling triangle}, and to a tiling triangle that lies above its hypotenuse as an {\it upper tiling triangle}. In any right triangle tiling, $\alpha$ is the smallest angle in the right triangle, and is opposite the horizontal edge. 

\begin{lemma}
\label{unbounded}
In a right triangle tiling, if a trajectory never meets two perpendicular edges in a row, then the trajectory is unbounded.
\end{lemma}

\begin{proof}
Consider a trajectory that never meets two perpendicular edges in a row as it passes through a lower tiling triangle, then meets the hypotenuse (Figure \ref{upandright} (a)). The trajectory then passes through the upper tiling triangle and crosses either the top or right edge. Either way, it will enter another lower tiling triangle, meet the hypotenuse, then go up or right once again. So the trajectory always travels up and right (or in the symmetric case, down and left) and escapes to infinity.
\end{proof}

\begin{figure}[h!]
\centering
\includegraphics[height=200pt]{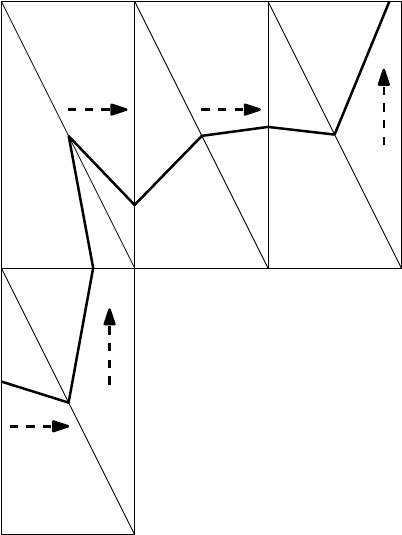} \ \ \ \ \ \ \ \ \ \ \ \ \ \ \ \ 
\includegraphics[height=200pt]{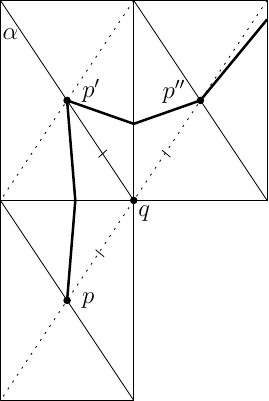}
\caption{(a) A trajectory that never meets two perpendicular edges in a row only travels up and to the right. (b) If a trajectory bisects a hypotenuse, it will bisect every hypotenuse.}
\label{upandright}\label{alwaysbisects}

\end{figure}

\begin{theorem}
\label{bisecting}
Every right triangle tiling has an escaping trajectory.
\end{theorem}

\begin{proof}
First, we show that if a trajectory bisects a hypotenuse, then it bisects the hypotenuse of every tiling triangle it enters. Then, we show that any trajectory that bisects a hypotenuse escapes.

Let $p$ be the midpoint of a hypotenuse, $p'$ be the reflection of $p$ across the horizontal edge above $p$, and $p''$ be the reflection of $p'$ across the vertical edge to the right of $p'$ (see Figure \ref{alwaysbisects} (b)). By symmetry, $p'$ is the midpoint of the hypotenuse, and also by symmetry, a trajectory through $p$ that crosses the horizontal edge must pass through $p'$. By a similar argument, $p''$ is the midpoint of the hypotenuse and a trajectory through $p'$ that crosses the vertical edge must pass through $p''$.

Hence any trajectory through a midpoint cannot meet two perpendicular edges in a row; by Lemma \ref{unbounded}, such a trajectory escapes.
\end{proof}

In Theorem \ref{iso}, we showed that in an isosceles triangle tiling, \emph{every} trajectory is periodic or drift-periodic. We now show (Theorem \ref{dper}) that some rational right triangle tilings always have \emph{at least one} drift-periodic trajectory, and we conjecture (Conjecture \ref{everyrightdp}) that this is true for all rational right triangle tilings.

\begin{theorem}
If $\alpha = \frac{\pi}{2n}$ for some $n \in \mathbb{N}$, then a drift-periodic trajectory exists.
\label{dper}
\end{theorem}

\begin{proof}
Construct a trajectory that perpendicularly bisects the short leg of a right tiling triangle (Figure \ref{rotationsymm}). The trajectory bisects the hypotenuse of the triangle, meeting it at angle $\alpha$. By the Angle Adding Lemma, there exists a hypotenuse that the trajectory meets at angle $(2n-1)\alpha = \pi - \alpha$. Since the trajectory bisects a hypotenuse, by Proposition \ref{bisecting}, the trajectory bisects every hypotenuse it meets, so the trajectory perpendicularly bisects the short leg of the upper triangle whose hypotenuse the trajectory meets at angle $(2n-1)\alpha$. Since the trajectory meets an edge at a corresponding point and at the same angle as where it started, the trajectory is drift-periodic. 
\end{proof}

Note that the trajectory will also bisect the long leg of a right tiling triangle so we could also begin our construction there, replacing $\alpha$ with $\frac{\pi}{2} - \alpha$ and following the same argument.

\begin{figure}[h!]
\centering
\includegraphics[scale=.70]{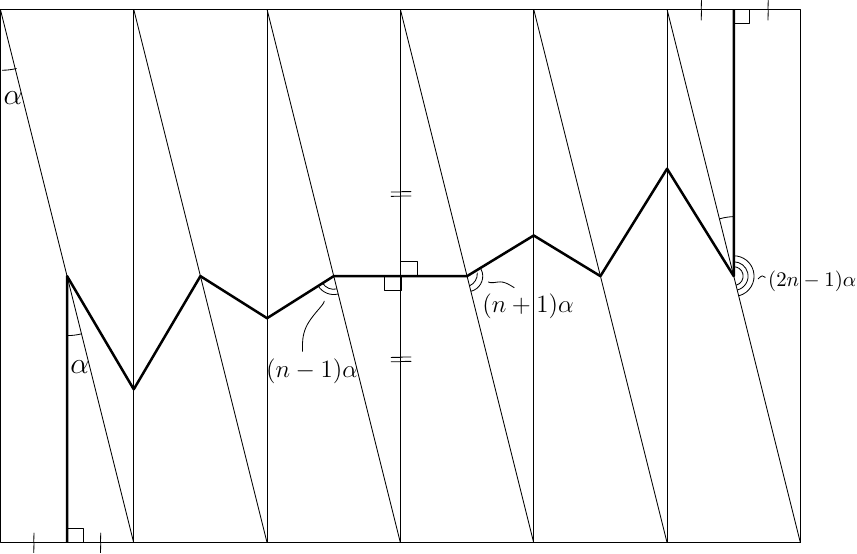}
\caption{A trajectory that perpendicularly bisects the legs of a right tiling triangle must be drift-periodic if $\alpha = \frac{\pi}{2n}$, $n \in \mathbb{N}$.}
\label{rotationsymm}
\end{figure}

In fact, our experiments suggest that Theorem \ref{dper} holds for all rational right triangles:

\begin{conjecture} \label{everyrightdp}
Every rational 
right triangle tiling has a drift-periodic trajectory.
\end{conjecture}


Many questions remain open regarding periodic orbits of right triangle tilings. During computer experimentation, we noticed that there appears to be a bound on the period of an orbit contained in a single row of a tiling (Figure \ref{1row}). All periodic orbits observed with larger periods crossed more than one row and remained in each row for the same number of refractions ($\pm 2$) as occur in the maximum one-row periodic orbit. However, not every possible number of rows was crossed in a periodic orbit. We would like to find a rule for how many rows will be crossed in a periodic orbit, and determine  if there is a bound on the period in a right triangle tiling.



\begin{figure}[h!]
\centering
\includegraphics[scale=0.3]{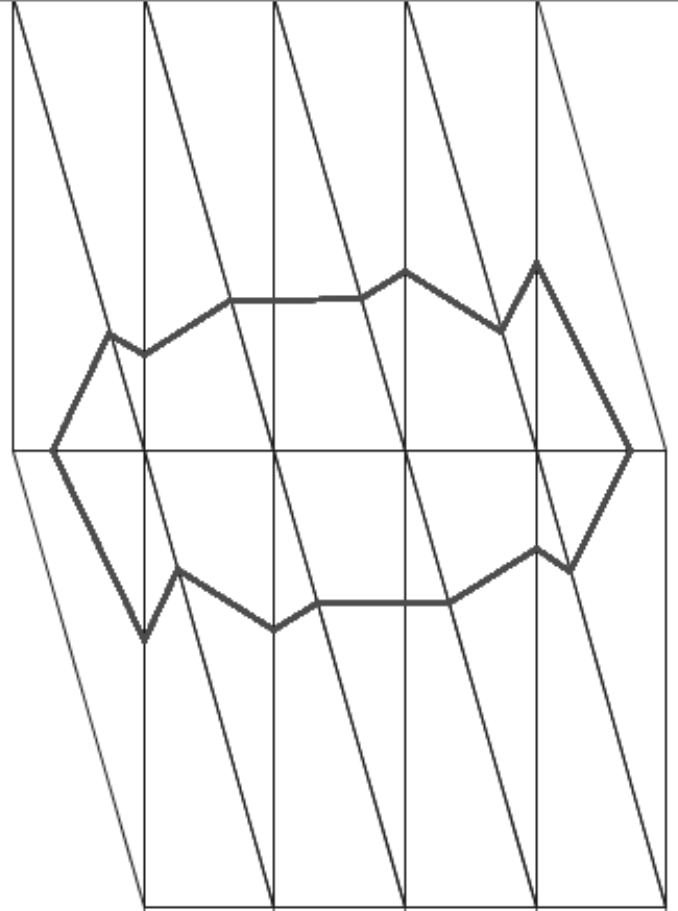} \hspace{0.15in}
 \includegraphics[width=165pt]{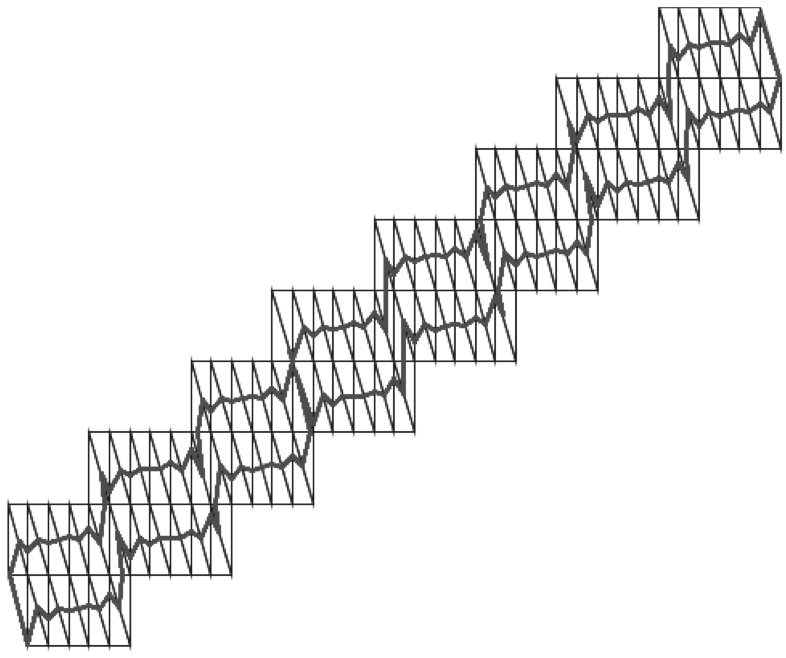}
\caption{(a) The largest periodic orbit contained in a single row. \label{1row} (b) A longer periodic orbit. Note that each row resembles the one-row periodic orbit. \label{manyrow}}
\end{figure}

\subsection{Periodic trajectories on general triangle tilings}

\begin{theorem} \label{per10}
Every triangle tiling, except tilings of isosceles triangles with vertex angle greater than or equal to $\frac{\pi}{3}$, has a periodic trajectory with period 10.
\end{theorem}

\begin{proof}

Let $\alpha$, $\beta$ be two angles of the tiling triangle and $\theta$ be the initial angle the trajectory makes with the side of the tiling triangle between $\alpha$ and $\beta$, on the $\alpha$ side of the trajectory. In Figure \ref{10periodic}, we see a ten-periodic trajectory for a generic triangle tiling. If each of the labeled angles is between $0$ and $\pi$, then there exists a trajectory around two groups of intersecting lines making the angles $\alpha$, $\beta$, and $\pi - \alpha - \beta$. When we add edges so that these intersecting lines become a tiling of triangles with angles $\alpha$, $\beta$, and $\pi - \alpha - \beta$, we can shrink the periodic trajectory with period $10$ so that it fits within the bounds of the triangles. 

\begin{figure}[h!]
\centering
\includegraphics[scale=1.25]{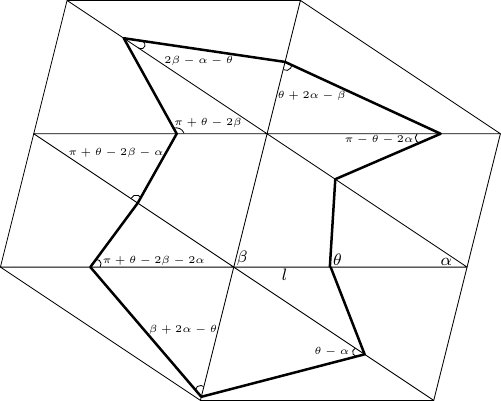}
\caption{If each of the labeled angles is positive, then this ten-periodic trajectory exists.}
\label{10periodic}
\end{figure}

If we start with a system of inequalities based on every angle in the trajectory being positive, we can reduce it to the following system, which implies that a ten-periodic trajectory exists:


$\begin{cases}
0&<\theta\\
\beta-2\alpha&<\theta\\
2\beta + 2\alpha - \pi&<\theta\\
\alpha &< \theta 
\end{cases}$ \hspace{1in}
$\begin{cases}
\theta < \pi \\
\theta < \pi-2\alpha\\
\theta < 2\beta-\alpha\\
\theta < \beta + 2\alpha. 
\end{cases}$

By combining each inequality on the left with each inequality on the right, we can reduce the system to the following three inequalities that depend only on $\alpha$ and $\beta$, and not on $\theta$: 

$\begin{cases}
\pi &> \beta + 2\alpha \\
\frac{\pi}{3} &> \alpha \\
\beta &> \alpha.
\end{cases}$


We graph these in Figure \ref{inequalities}. The region of values of $\alpha$ and $\beta$ where there exists a periodic trajectory of period $10$ is shaded dark gray, and does not contain its boundaries. Since this region contains all 
scalene triangles,
every scalene triangle tiling has a periodic trajectory of period $10$. However, this is not so for every isosceles triangle tiling. Isosceles triangles lie on the lines $\alpha = \beta$, $\pi = 2\alpha + \beta$ (both dashed), and $\pi = 2\beta + \alpha$ (dotted). The two dashed  lines coincide with the boundaries of the region where there exists a periodic trajectory of period $10$. The third and dotted line is contained in the region of acceptable values when $\pi/3 < \beta < \pi/2$ and $\alpha < \pi/3$. In other words, an isosceles triangle only has a ten-periodic trajectory if its vertex angle is less than $\frac{\pi}{3}$, and a base line of the tiling must cross the interior of the periodic trajectory of period $10$.

\begin{figure}[h!]
\centering
\includegraphics[width=300pt]{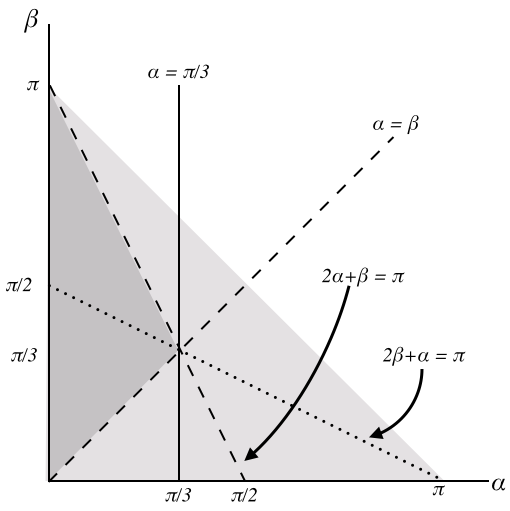}
\caption{The system of inequalities $\pi > \beta + 2\alpha, \pi > 2\beta + \alpha$, and $\beta > \alpha$. The (open) gray region represents all possible triangles. The (open) dark gray region represents all triangles with a period-$10$ trajectory. The dashed and dotted lines represent isosceles triangles. The portion of the dotted line \mbox{$\pi = 2\beta + \alpha$} in the dark gray region represents the set of $\alpha$, $\beta$ where an isosceles triangle has a $10$-periodic trajectory.}
\label{inequalities}
\end{figure}

The above shows that the trajectory returns to the same angle. We can show that the location (distance along the edge) is also the same via repeated application of the Law of Sines.


We can construct a ten-periodic trajectory by choosing a value for $l$, the distance of the trajectory from the vertex of angle $\beta$ along the side between $\alpha$ and $\beta$ where the trajectory makes the angle $\theta$ (Figure \ref{10periodic}), such that our system of inequalities holds.
%

For example, if $\alpha = \frac{\pi}{5}$ and $\beta = \frac{3\pi}{10}$, then $\theta$ must be between $\frac{\pi}{5}$ and $\frac{2\pi}{5}$. Let $\theta = \frac{3\pi}{10}$. Then $0 < l < \frac2{3+\sqrt5}$, when the edge of the tiling triangle between angles $\alpha$ and $\beta$ has length 1.


\end{proof}


\begin{conjecture}
There exist triangle tilings with periodic trajectories of arbitrarily large length.
\end{conjecture}

For example, the scalene triangle tiling in  Figure \ref{34periodic} has a periodic trajectory of period $34$. In fact, all the periods we have observed follow a pattern:

\begin{figure}[h!]
\centering
\includegraphics[scale=0.45]{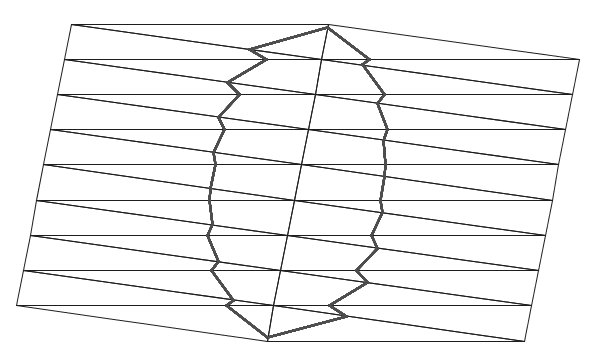}
\caption{The triangle tiling with angles $8^{\circ}, 79^{\circ},$ and $93^{\circ}$ and a trajectory of period $34$.}
\label{34periodic}
\end{figure}

\begin{conjecture} \label{fournplustwo}
In a triangle tiling, every periodic trajectory has a period of the form $4n+2$, for $n\geq 1$.
\end{conjecture}

In Corollary \ref{periodic_unstable}, we showed that for a division of the plane by an odd number of lines, a trajectory can spiral. We do not observe this behavior on triangle tilings:

\begin{conjecture} \label{trinospiral}
Trajectories on triangle tilings never spiral.
\end{conjecture}

\begin{conjecture}\label{trinodense}
Trajectories on triangle tilings never fill a region of the plane densely.
\end{conjecture}

\section{The Trihexagonal Tiling} \label{kelsey}

The equilateral triangle tiling and the regular hexagon tiling have very simple dynamics (Theorem \ref{mf}), but the composition of these two tilings into the trihexagonal tiling offers a myriad of interesting dynamics, which we explore in this section. As mentioned in the Introduction, trajectories in the trihexagonal tiling are very unstable; a tiny change in the trajectory direction can transform a periodic trajectory into a trajectory that escapes, or into one that ``fills'' a region of the plane (in the sense of filling in all the pixels on the screen to some level of resolution in our program). In this sense, the trihexagonal tiling has similar behavior to inner billiards on the square table, where a tiny change in the direction of a trajectory, from a rational to an irrational slope, transforms a periodic trajectory into a dense trajectory.

Our goal was to find periodic trajectories and drift-periodic trajectories on the trihexagonal tiling. The dynamics of this tiling turn out to be complicated and interesting, so we analyze the local behavior of trajectories around single vertices or tiles (Lemmas \ref{turner}-\ref{pent_lemma}), and then give explicit examples of periodic trajectories on this tiling (Examples \ref{6theorem}, \ref{12periodic}, and \ref{gooseheadtheorem}). We also construct two families of drift-periodic trajectories (Propositions \ref{drift120} and \ref{turning_drift}), and for one of these families, we show that the limiting trajectory fills infinite regions of the plane with parallel segments that are arbitrarily close together (Theorem \ref{dense}).

\begin{definition}
The \emph{trihexagonal tiling} is the edge-to-edge tiling where an equilateral triangle and a regular hexagon meet at each edge. Examples showing large areas of this tiling are in Figures \ref{bigperiods1}-\ref{bigperiods2}. We assume that each of the edges in the tiling has unit length. 

We assume that a trajectory $\tau$ starts on an edge of the tiling, at a point $p_1$, and its subsequent intersections with edges are $p_2, p_3$, etc. The distance $x_i$ from $p_i$ to a vertex is not clearly defined; it could be $x_i$ or $1-x_i$. It is most convenient for us to use triangles, so we choose to define $x_i$ as follows: Every edge of the trihexagonal tiling is between an equilateral triangle and a hexagon. In the equilateral triangle, $\tau$ goes from $p_i$ on edge $e_i$ to $p_{i+1}$ on an adjacent edge $e_{i+1}$, thus creating a triangle whose edges are formed by $e_i, e_{i+1}$ and $\tau$. Let $x_i$ be the length of the triangle's edge along $e_i$ and let $x_{i+1}$ be the length of the triangle's edge along $e_{i+1}$ (see Figures \ref{turningpicture}-\ref{quadtriangle}).

Define the angle $\alpha_i$ to be the angle between $\tau$ and an edge of the tiling at $p_i$, where $\alpha_i$ is between the trajectory and the edge with length $x_i$ as defined above. The \emph{initial angle} of a trajectory is the angle $\alpha_1$; since we frequently use this angle, we denote it by $\alpha$. 

\end{definition}

\subsection{Local geometric behavior in the trihexagonal tiling}

First, we state four lemmas in elementary geometry about local trajectory behavior in the tiling. In the next section, we use these lemmas to prove periodicity of several periodic and drift-periodic trajectories in the tiling.

\begin{lemma}[The Trajectory Turner]
\label{turner}
Consider a trajectory crossing the edges of the tiling at ${p_1},{p_2},{p_3},{p_4}$ where segments $\overline{{p_1}{p_2}}$ and $\overline{{p_3}{p_4}}$ lie in distinct triangles, and segment $\overline{{p_2}{p_3}}$ lies in the hexagon adjacent to both triangles. 
Then ${x_1} = {x_4}$. Also, $\alpha_4=\pi-\alpha$.
\end{lemma}

\begin{proof} 
We include a proof of this elementary result because we can use a folding argument (see Theorem \ref{mf} in the Introduction) even though the tiling is \emph{not} reflection-symmetric across the edges of the tiling.

In Figure \ref{turningpicture}, we fold the triangles onto the hexagon, and find that lengths $x_1$ and $x_4$ measure distances that fold up to the same segment, so they are equal. 
The angles at $p_1$ and $p_4$ fold up to make a straight line, so $\alpha_4 = \pi-\alpha_1$.
\end{proof}
%
\begin{figure}[h!]
\centering
\includegraphics[width=380pt]{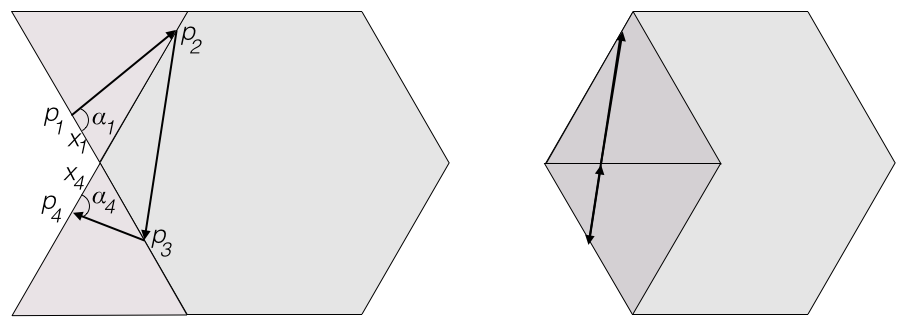}
\caption{The Trajectory Turner, proven via a folding argument}
\label{turningpicture}
\end{figure}
%
%

\begin{lemma}[The Quadrilateral]
\label{quadrilateral}
In a regular hexagon $ABCDEF$, suppose a trajectory passes from side $AB$ to side $CD$, crossing $AB$ at ${p_1}$ and $CD$ at ${p_2}$. 
If the angle at ${p_1}$, measured on the same side of the trajectory as $B$, is $\pi-\alpha$,  then \mbox{${x_2}$ = $\frac{\sin(\alpha)(2{x_1}+1) + \sqrt{3}\cos(\alpha)}{2\sin(\alpha - \frac{\pi}{3})}$}. Also, $\alpha_2=\alpha-\pi/3$.
\end{lemma}

\begin{figure}[h!]
\center
\includegraphics[width=200pt]{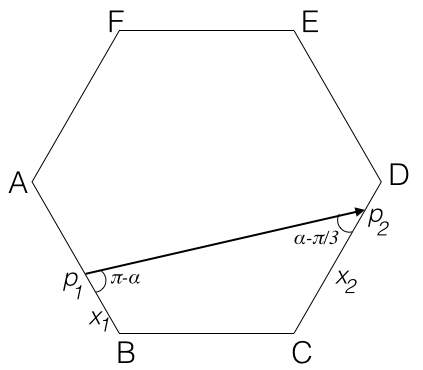}
\caption{The Quadrilateral}
\label{quadzoom}
\end{figure}

%
%
%
%
%
%
%
%

\begin{lemma}[The Quadrilateral-Triangle]
\label{quad_triangle}
Let ABCDEF, ${p_1}$, ${p_2}$ and $x_1$ be as in The Quadrilateral. Consider equilateral triangle CDT, and let ${p_3}$ be the point where the trajectory meets CT (Figure \ref{quadtriangle} (a)). 

Then ${x_3} = \frac{1}{2} + {x_1} + \frac{\sqrt{3}}{2}\cot\alpha$. Also, $\alpha_3=\pi-\alpha$.
\end{lemma}

\begin{lemma}[The Pentagon]
\label{pent_lemma}
In a regular hexagon $ABCDEF$, suppose a trajectory passes from side $AB$ to side $DE$ , crossing $AB$ at ${p_1}$ and $DE$ at ${p_2}$ (Figure \ref{quadtriangle} (b)). If the angle at ${p_1}$, measured on the same side of the trajectory as $B$, is ${\pi - \alpha}$, then ${x_2} = {x_1} + \sqrt{3}\cot\alpha$. Also, $\alpha_2=\alpha$.
\end{lemma}

%
%
%
%
%
%
%

\begin{figure}[h!]
\center
\includegraphics[height=140pt]{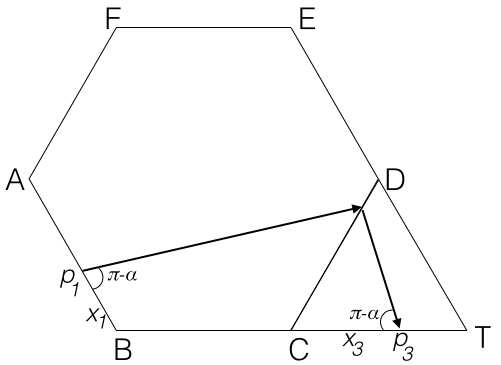} \ \ \ \ \  \ \ \ \ \ \  \ 
\includegraphics[height=140pt]{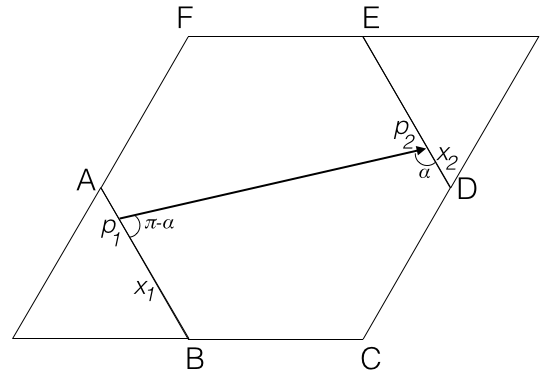}
\caption{(a) The Quadrilateral-Triangle (b) The Pentagon}
\label{quadtriangle}
\end{figure}

%
%

\subsection{Periodic and drift-periodic trajectories in the trihexagonal tiling}

Using our results about local behavior in the trihexagonal tiling, we will now prove the existence of specific periodic and drift-periodic trajectories. In each case, we first give a simple periodic trajectory (Examples \ref{6theorem} and \ref{12periodic}), and then show how a perturbation of the simple trajectory yields a family of drift-periodic trajectories with arbitrarily large period (Propositions \ref{drift120} and \ref{turning_drift}).

\begin{example}
\label{6theorem}
There is a 6-periodic trajectory intersecting all edges of the tiling at angle $\frac{\pi}{3}$ (Figure \ref{period6}). This trajectory circles three lines forming a regular triangle, so by Corollary \ref{ngoncorollary}, an initial angle of $\frac{3-1}{2\cdot 3}\pi = \frac{\pi}{3}$ makes the trajectory periodic. 

This trajectory is stable under parallel translations for any $0 < {x_1} < 1$; see the dashed trajectory in Figure \ref{period6}.  
\end{example}

\begin{figure}[h!]
\center
\includegraphics[height=140pt]{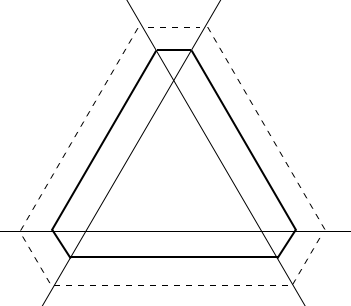} \ \ \ \ \ \ \ \ \ \ 
\includegraphics[height=140pt]{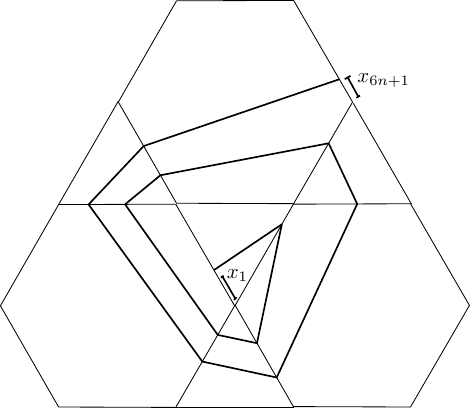}
\caption{(a) The period-6 trajectory of Example \ref{6theorem}, (b) One period of the drift-periodic trajectory in Proposition \ref{drift120} where $n = 2$}
\label{period6}\label{sides}
\end{figure}

If we perturb the periodic trajectory in Figure \ref{period6} (a), we obtain a family of drift-periodic trajectories. One period  of such a trajectory is in Figure \ref{sides} (b). We describe this family in Proposition \ref{drift120}.

\begin{proposition}
\label{drift120}
For any $n \in N$, there is a drift-periodic trajectory of period $6n$ that intersects each edge of the tiling a maximum of $n$ times. 
\end{proposition}

\begin{proof}

Let \mbox{$\alpha = \pi - \tan^{-1}\left(\frac{3n\sqrt{3}}{3n-2}\right)$}. (This angle is less than, and approaches, $2\pi/3$.) Then $$\mbox{$\cot\alpha = \frac{-(3n-2)}{3n\sqrt{3}}=\frac{-\frac 12 (3n-2)}{(3n-2)\frac{\sqrt{3}}2+\sqrt{3}}.$}$$

We wish to show that $x_{6n+1}=x_1$ and $\alpha_{6n+1}=\alpha$.

By the Pentagon Lemma, $x_{6n+1}=x_{6n}+\sqrt{3}\cot(\pi-\alpha_{6n})$ and $\alpha_{6n+1}=\pi-\alpha_{6n}$.

By the Quadrilateral-Triangle Lemma, $x_{2k+2}=x_{2k}+\frac 12 + \frac{\sqrt{3}}2 \cot (\pi-\alpha_{2k})$  and $\alpha_{2k+2}=\alpha_{2k}$ for $k\geq 3$. We apply this result $3n-2$ times to $x_{6n}$ and $\alpha_{6n}$ to yield $\alpha_{6n+1}=\pi-\alpha_4$ and
\begin{align*}
x_{6n+1} &=x_{6n}+\sqrt{3}\cot(\pi-\alpha_{6n}) \\
&= x_4+(3n-2)\left(\frac 12 + \frac{\sqrt{3}}2 \cot (\pi-\alpha_4) \right)+\sqrt{3}\cot(\pi-\alpha_4).
\end{align*}
By the Trajectory Turner Lemma, $x_4=x_1$ and $\alpha_4=\pi-\alpha_1=\pi-\alpha$, so $\alpha_{6n+1}=\alpha$ and
\begin{align*}
x_{6n+1} &= x_1+(3n-2)\left(\frac 12 + \frac{\sqrt{3}}2 \cot \alpha \right)+\sqrt{3}\cot \alpha \\
&= x_1+\frac 12 (3n-2) + \cot\alpha  \left( (3n-2) \frac {\sqrt{3}}2 + \sqrt{3}\right) = x_1,
\end{align*}
by substituting the expression for $\cot\alpha$.

\end{proof}


\begin{remark}
We can take the limit of the initial angle $\alpha$ as $n\to\infty$: \[\lim_{n \to +\infty} \alpha = \lim_{n \to +\infty} \pi - \tan^{-1}\left(\frac{3n\sqrt{3}}{3n-2}\right) = \frac{2\pi}{3}.\] This implies that our drift-periodic trajectory is converging to the periodic trajectory in Example \ref{6theorem}. Our angle $\alpha$ is bounded below by $\alpha = \pi - \tan^{-1}(3\sqrt{3})\approx 0.56\pi$ and above by $\alpha = \frac{2}{3}\pi$. 
\end{remark}

\begin{example}
\label{12periodic}
There is a 12-periodic trajectory that intersects the edges of the tiling at  angles $\frac{\pi}{2}$ and $\frac{\pi}{6}$ (Figure \ref{period12} (a)).

\begin{figure}[h!]
\center
\includegraphics[width=190pt]{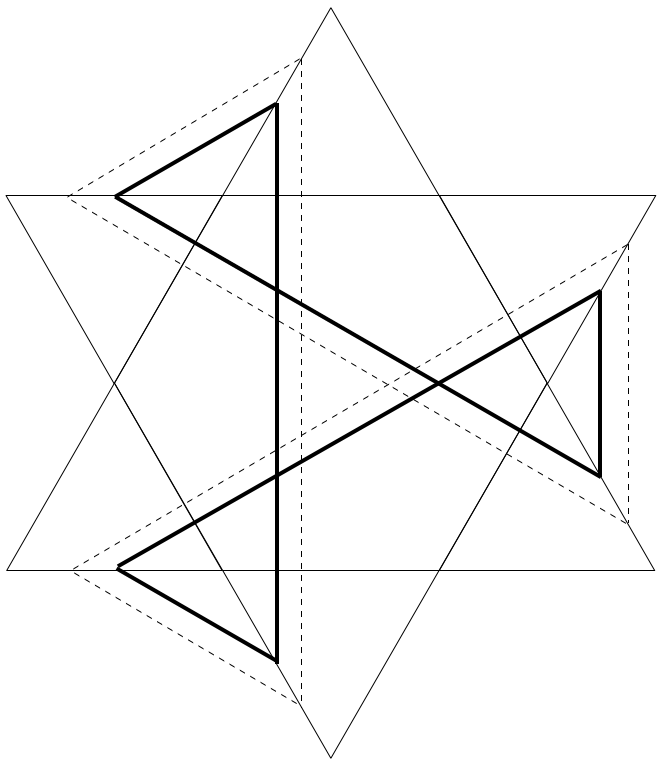}  \ \ \ \ \ 
\includegraphics[scale = 0.73]{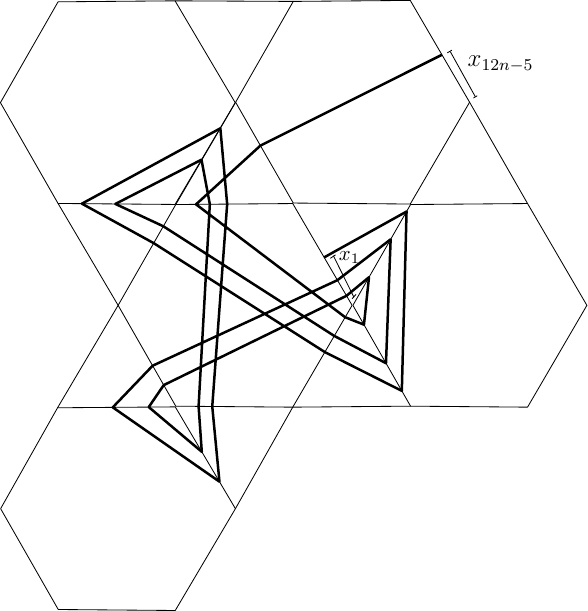}
\caption{(a) Two examples of the period-12 trajectory of Example \ref{12periodic}, one thick and one dashed (b) One period of the associated drift-periodic trajectory, where $n = 3$}
\label{period12}\label{drift_tri}
\end{figure}

%
%

This trajectory is stable under parallel translations:  As long as the trajectory intersects the edge of the central hexagon perpendicularly, any value of ${x_1}$, with $0 < {x_1} < \frac{1}{2}$, produces a parallel periodic trajectory. An example is shown as a dashed trajectory in Figure \ref{period12} (a).
\end{example}

\begin{proposition}
\label{turning_drift}
For any $n \geq 2$, there is a drift-periodic trajectory with period $12{n} - 6$ that intersects each edge of a tiling a maximum of $n$ times (Figure \ref{driftconverge}). 
\end{proposition}


\begin{proof}

Given $n$, choose $\alpha = {\tan}^{-1}((6n-3)\sqrt{3})$ as the initial angle of the trajectory. (This angle is less than, and approaches, $\pi/2$.)
In order to avoid hitting a vertex, we need to carefully choose ${p_1}$. The following calculations show that given any $n$, there is an $x_1>0$ so that ${p_1}$ can be placed a distance $x_1$ from the midpoint and create the drift-periodic trajectory. 

By the Pentagon Lemma, we calculate $x_{13} = x_1 + 3\sqrt{3}\cot\alpha$. This gives $|p_{1}p_{13}|=3\sqrt{3}\cot\alpha$, and in general, since $\alpha = {\tan}^{-1}((6n-3)\sqrt{3})$, we have
\begin{equation}\label{intersections}
|p_{12k+i}p_{12(k+1)+i}|=3\sqrt{3}\cot\alpha=\frac{3\sqrt{3}}{(6n-3)\sqrt{3}}=\frac 1{2n-1}.
\end{equation}
If there are $n$ intersections to an edge, there are $n-1$ of these distances, so the total distance is 
$\frac{n-1}{2n-1}$, 
which is less than $1/2$, 
as desired. 

Now that we have the initial angle and initial starting point of the trajectory, we will show that the trajectory is drift-periodic,
by showing that $x_1={x_{12n-5}}$ (Figure \ref{drift_tri} (b)).

By the Trajectory Turner Lemma, $$\alpha_{12k+4}=\pi-\alpha_{12k+1}, \ \ \  \alpha_{12k+8}=\pi-\alpha_{12k+5}, \text{\ \ and \ \  } \alpha_{12k+12}=\pi-\alpha_{12k+9}.$$ 

By the Pentagon Lemma, $$\alpha_{12k+5}=\pi-\alpha_{12k+4}, \ \ \  \alpha_{12k+9}=\pi-\alpha_{12k+8}, \text{\ \  and \ \ }  \alpha_{12k+13}=\pi-\alpha_{12k+12}.$$ 
Also, we know $\alpha_1=\alpha$. Combining these yields $$\alpha_{12k+1}=\alpha_{12k+5}=\alpha_{12k+9}=\alpha, \text{\ \  and \ \ } \alpha_{12k+4}=\alpha_{12k+8}=\alpha_{12k}=\pi-\alpha.$$
Also by the Trajectory Turner Lemma, $$x_{12k+1}=x_{12k+4}, \ \ \  x_{12k+5}=x_{12k+8} , \text{\  \ and \ \  } x_{12k+9}=x_{12k+12}$$ for all $k$. 

By the Pentagon Lemma, 
\begin{align*}
x_{12k+5} &=x_{12k+4}+\sqrt{3}\cot(\pi-\alpha_{12k+4}) = x_{12k+4}+\sqrt{3}\cot(\alpha); \\
x_{12k+9} &=x_{12k+8}+\sqrt{3}\cot(\pi-\alpha_{12k+8}) = x_{12k+8}+\sqrt{3}\cot(\alpha); \\
x_{12k+13} &=x_{12k+12}+\sqrt{3}\cot(\pi-\alpha_{12k+12}) = x_{12k+12}+\sqrt{3}\cot(\alpha).
\end{align*}
Combining these yields $x_{12k+13}=x_{12k+1}+3\sqrt{3}\cot\alpha$, so the distance between $p_{12k+13}$ and $p_{12k+1}$ is $3\sqrt{3}\cot\alpha$. There are $n-1$ of these gaps, so $x_{12n-11}=x_1+(n-1)3\sqrt{3}\cot\alpha$.

Now we apply the Trajectory Turner Lemma one more time: $$x_{12n-8}=x_{12n-11} , \text{\ \ and \ \ } \alpha_{12n-8}=\pi-\alpha_{12n-11}.$$
Now we apply the Quadrilateral Triangle Lemma and substitute in the relations from above: 
\begin{align*}
x_{12n-6} &=x_{12n-8}+\frac 12 + \frac {\sqrt{3}}2 \cot(\pi-\alpha_{12n-8}) \\
&=x_{12n-11} + \frac 12 + \frac {\sqrt{3}}2 \cot(\alpha_{12n-11}) \\
&=x_1+(n-1)3\sqrt{3}\cot\alpha + \frac 12 + \frac {\sqrt{3}}2 \cot\alpha.
\end{align*}
Finally, we apply the Pentagon Lemma and then substitute in the relations from above: 
\begin{align*}
x_{12n-5} &=x_{12n-6}+\sqrt{3}\cot(\pi-\alpha_{12n-6}) \\
&=x_{12n-6}+\sqrt{3}\cot\alpha \\
&=x_1+(n-1)3\sqrt{3}\cot\alpha + \frac 12 + \frac {\sqrt{3}}2 \cot\alpha+\sqrt{3}\cot\alpha \\
&=x_1+ \frac 12 +\sqrt{3}\left(3n-3+1-\frac 12\right)\cot\alpha \\
&=x_1+ \frac 12 +\frac{\sqrt{3}\left(3n-3+1-1/2\right)}{-\sqrt{3}(6n-3)} = x_1+\frac 12 - \frac 12 = x_1,
\end{align*}
 as desired.

%
%
%
%

%
%
%
%
%
\end{proof}

\begin{theorem}\label{dense}
The trihexagonal tiling exhibits trajectories that fill infinite regions of the plane with line segments that are arbitrarily close together.
\end{theorem}

\begin{proof}
By (\ref{intersections}), the intersections of the $n^{\text{th}}$ member of the family of drift-periodic trajectories described in Proposition \ref{turning_drift} with a selected edge are a distance $\frac 1 {2n-1}$ apart. Thus, we can construct a trajectory where the intersections are arbitrarily close together, and thus the line segments on the neighboring tiles, and then on infinitely many tiles, are also arbitrarily close.
\end{proof}

The word ``trajectories'' is plural in Theorem \ref{dense} because we expect that a similar analysis on other families of drift-periodic trajectories would yield the same result; however, the required calculations are onerous and we have not undertaken them.

\begin{figure}[h!]
\centering
	\includegraphics[scale=0.3]{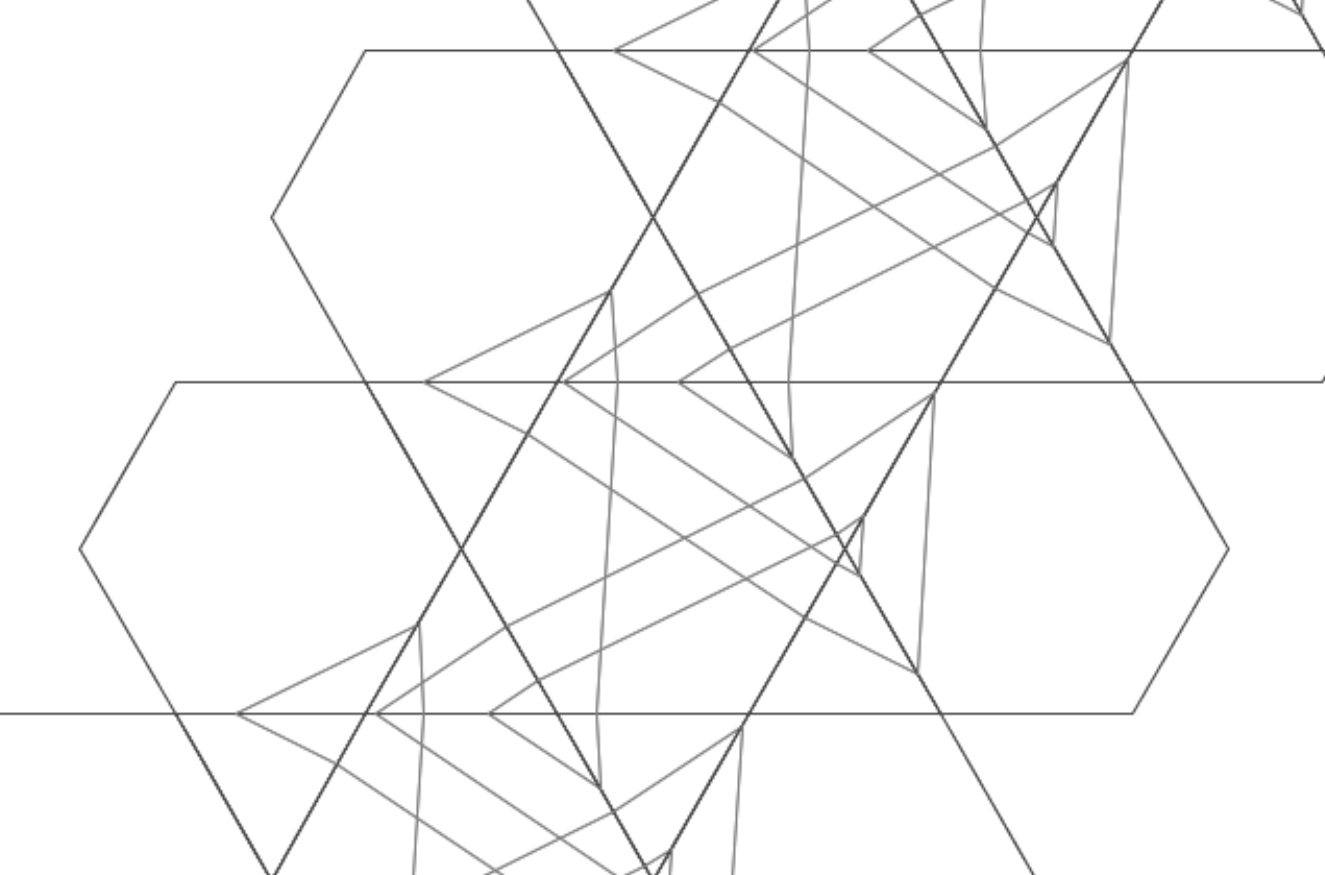}
	\includegraphics[scale = 0.3]{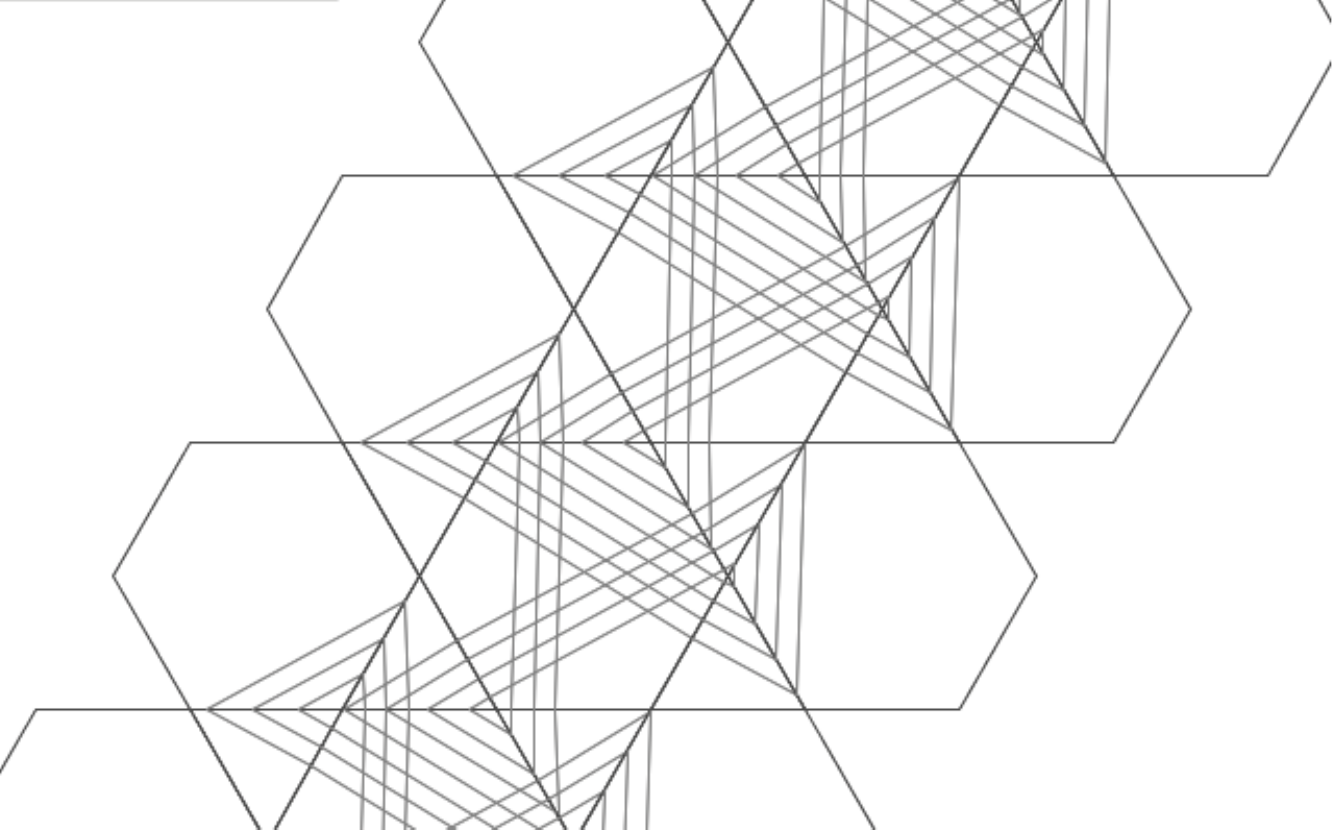}
\caption{A drift-periodic trajectory with period (a) 18 and (b) 42, which are the $n=2$ and $n=4$ cases, respectively, of Proposition \ref{turning_drift}.}
\label{driftconverge}
\end{figure}

Figure \ref{driftconverge} shows two examples of the family of trajectories in Proposition \ref{turning_drift}; one can extrapolate what the limiting trajectory would look like. For some hexagons in the tiling, this limiting trajectory fills one part with arbitrarily close segments but does not visit another part of the hexagon at all. This is analogous to a billiard table in which a trajectory fills some region of the table densely and does not visit another part of the table at all, see \cite{mcmullen}.

\begin{conjecture} \label{denseexists}
There are trajectories that are dense in a region of the trihexagonal tiling.
\end{conjecture}

The periodic trajectories are perhaps the most aesthetically pleasing, and the rationality constraints of our computer program encourage us to explore periodic and drift-periodic trajectories. Still, we believe that the type of drift-periodic families discussed here are not the only escaping trajectories on this tiling:

\begin{conjecture} \label{nonperesc}
There are non-periodic escaping trajectories.
\end{conjecture}

\begin{remark}
As suggested by Figure \ref{driftconverge}, the drift-periodic trajectories of Proposition \ref{turning_drift} converge to the period-12 trajectory of Example \ref{12periodic} as $n \to \infty$. This shows that even though we can construct a trajectory whose segments are arbitrarily closely packed, in the limit this construction produces not a dense trajectory, but one that is periodic.

Also see the related periodic trajectories in Figure \ref{bigperiods1}. Considering that the period-12 trajectory in Figure \ref{period12} yields the nearby drift-periodic trajectories in Figure \ref{driftconverge}, it is possible that a perturbation of these larger trajectories and others like them   will yield nearby drift-periodic trajectories.
\end{remark}

\begin{figure}[h!]
	\includegraphics[height=200pt]{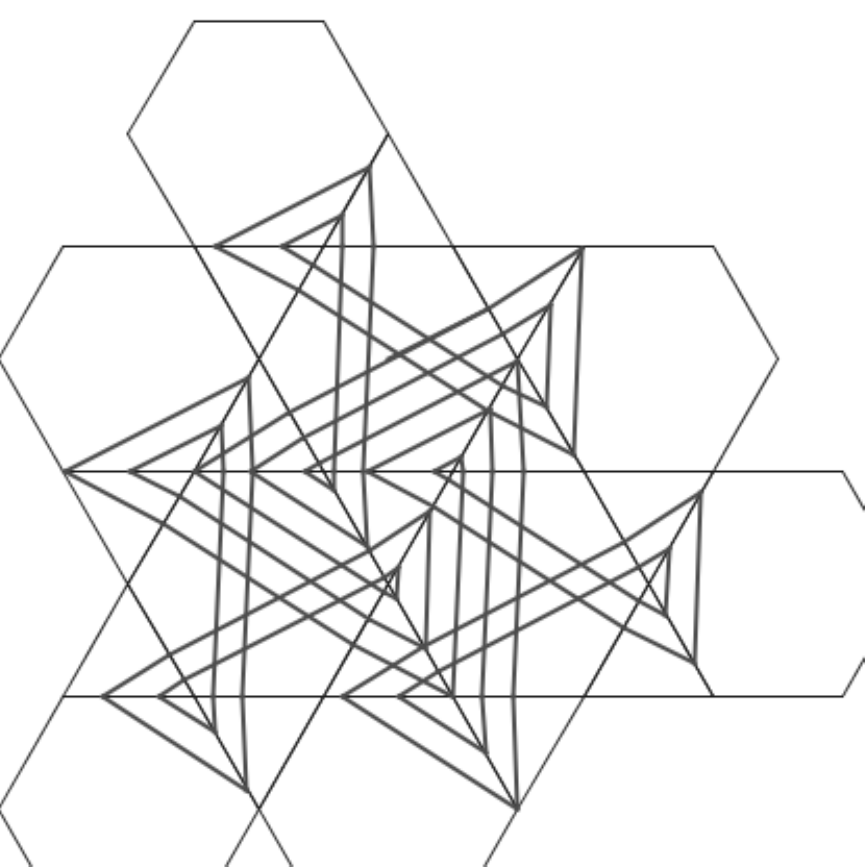}
	\includegraphics[height=200pt]{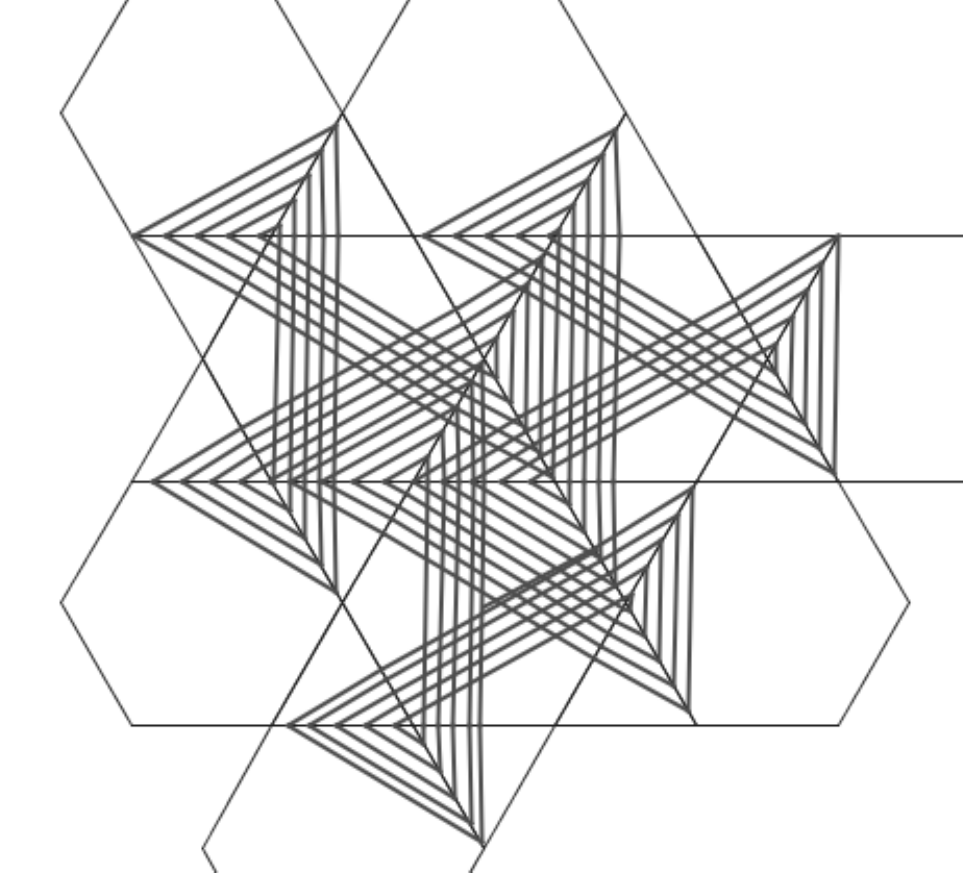}
	\caption{Trajectories with periods of 78 and 174}
	\label{bigperiods1}
\end{figure}

\begin{example}
\label{gooseheadtheorem}
There is a trajectory of period $24$, with an initial angle of $\alpha = \pi - \tan^{-1}(2\sqrt{3})$. 
\end{example} 

\begin{figure}[h!]
\center
\includegraphics[trim={0 0 0 1cm},clip=true,scale=0.75]{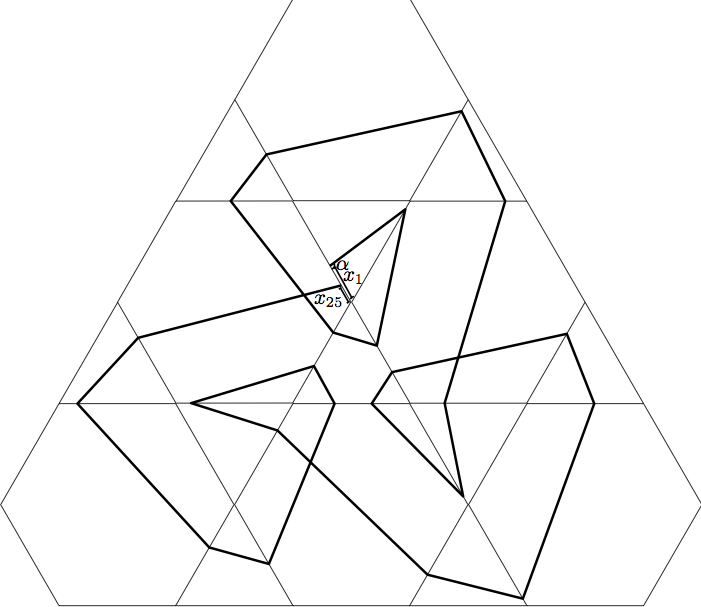}
\caption{The trajectory of period $24$ in Example \ref{gooseheadtheorem}}
\label{goosehead_pic}
\end{figure}

\begin{proof}
Define ${x_1}$ and $x_{25}$ as in Figure \ref{goosehead_pic}. 

%
%

By the Trajectory Turner Lemma, 
$x_{i+3}=x_i$ for $i=1,9,17$. This does not change the value of the distance $x_i$. 

By the Quadrilateral Triangle Lemma, $x_{i+2}=x_i+\frac 12 + \frac{\sqrt{3}}2 \cot\alpha$ for $i=4, 6,12, 14, 20, 22$. This results in adding $\frac 12 +\frac{\sqrt{3}}2 \cot\alpha$ to the distance $x_i$ a total of $6$ times, for a total increase of $3+3\sqrt{3}\cot\alpha$.

By the Pentagon Lemma, $x_{i+1}=x_i+\sqrt{3}\cot\alpha$ for $i=8, 16, 24$.  This results in adding $\sqrt{3}\cot\alpha$ to the distance $x_i$ a total of $3$ times, for a total increase of $3\sqrt{3}\cot\alpha$. 

Thus $x_{25}=x_1+3+6\sqrt{3}\cot\alpha$. The substitution $\cot\alpha = \frac{-1}{2\sqrt{3}}$ yields $x_{25}=x_1$, as desired.

\end{proof}

Also see the larger periodic trajectories in Figure \ref{bigperiods2} that resemble the trajectory in Figure \ref{goosehead_pic}.

\begin{figure}[h!]
	\includegraphics[trim={0 1cm 0 0.3cm},clip=true,height=200pt]{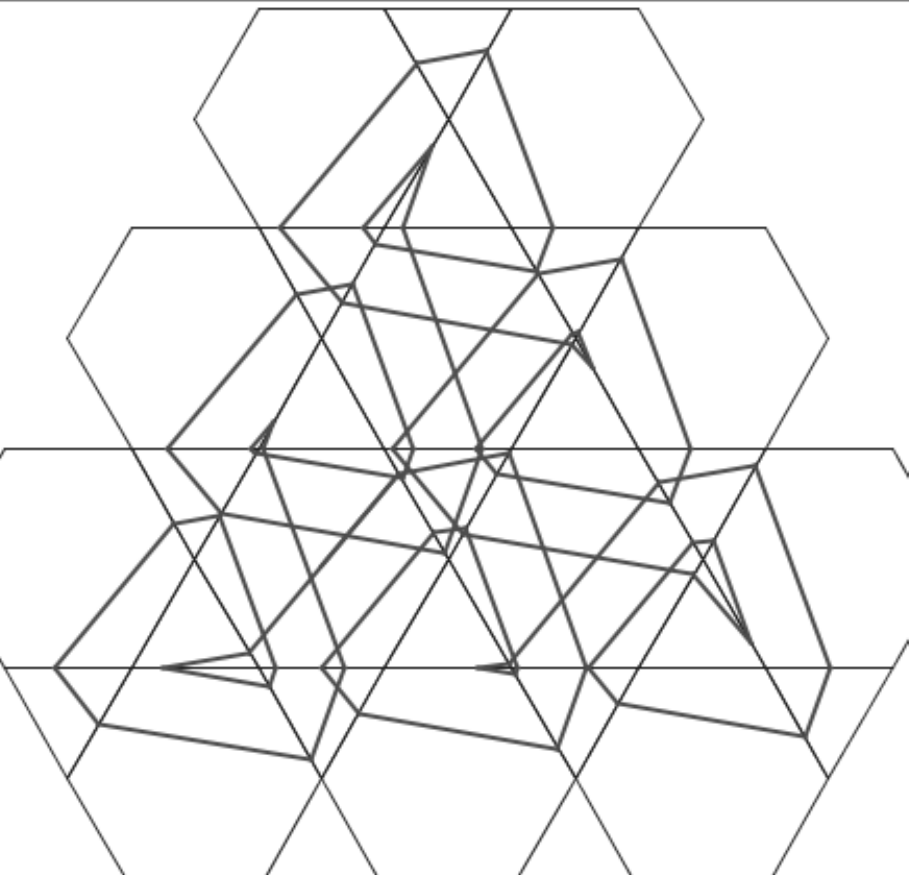}
	\includegraphics[height=200pt]{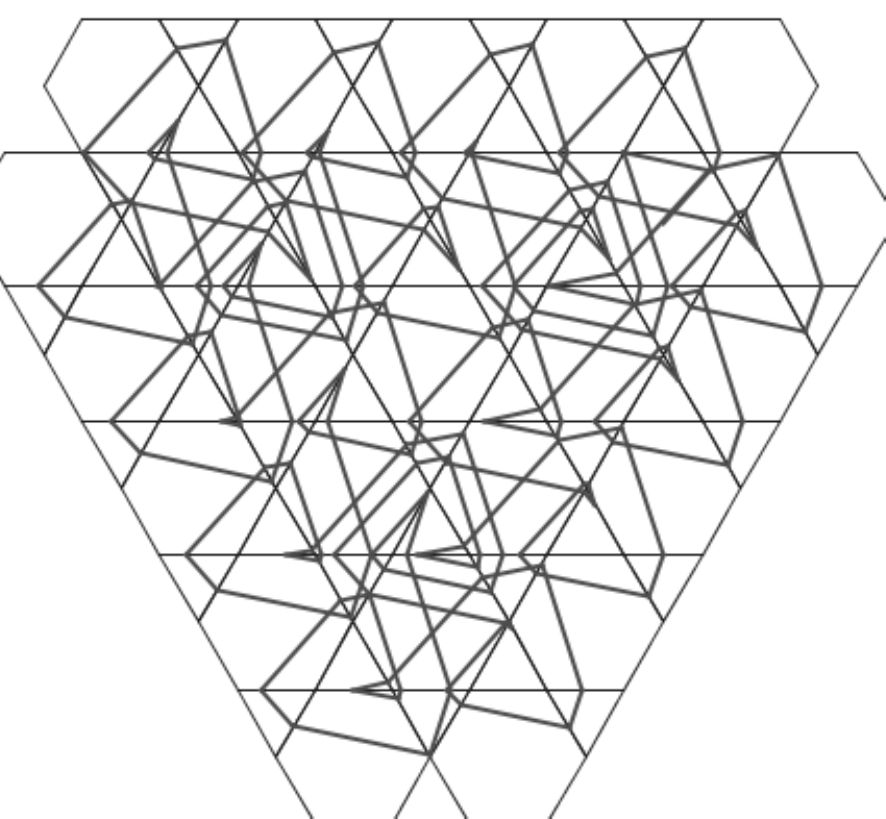}
	\caption{Trajectories with periods of 66 and 192 that resemble the trajectory in Figure \ref{goosehead_pic}}
	\label{bigperiods2}
\end{figure}

\newpage
\section{Future directions}

{\bf Finite line arrangements:} In this paper, we only studied divisions of the plane by finitely many non-parallel lines. One could relax these conditions, and study divisions of the plane with infinitely many lines, or include parallel lines. The unique feature of a division of the plane by lines is that there are regions of infinite area; one could study other tilings that have this property, such as a division of the plane by the curves $y=\sin(x)+k$ for $k\in\mathbf{Z}$, or related piecewise-linear curves.

{\bf Triangle tilings:} We have many conjectures about triangle tilings based on experimental results; these are described in Section \ref{jenny}. We  showed that all triangle tilings have a trajectory of period $6$ (Corollary \ref{all-triangles}) and that nearly all have a trajectory of period $10$ (Theorem \ref{per10}). Conjecture \ref{fournplustwo} says that all periodic trajectories on triangle tilings are of the form $4n+2$; we would like to know, for a given $n$, which triangle tilings have a trajectory of period $4n+2$. 

Our results on these triangle tilings are simple in the cases where the system is very stable. In contrast, the trihexagonal system is very unstable. We would like to understand what feature of a tiling causes this stability or instability. The triangle and trihexagonal tilings both have half-turn symmetry at each vertex, and in both cases the lines of the tiling are \emph{not} lines of reflective symmetry for the tiling, so this is not the explanation. In tilings of triangles where some lines \emph{are} lines of reflective symmetry for the tiling, such as isosceles triangle tilings and the regular tilings in Figure \ref{regtilings}, the behavior of trajectories is very predictable. 

Because the congruent triangle tilings we consider are those created by adding parallel diagonals to a parallelogram tiling, a next step would be to investigate this system on the parallelogram tiling itself. Another possibility is to consider edge-to-edge triangle tilings meeting $6$ to a vertex, where the orientation of every other row of triangles is opposite. Still another direction is to consider tilings of congruent triangles that are not edge-to-edge; the congruent triangle tilings that we consider are a two-parameter family, and allowing this offset adds an additional parameter. Each of these tilings is pictured in Figure \ref{triangle-tilings}.

\begin{figure}[h!]
\centering
\includegraphics[width=300pt]{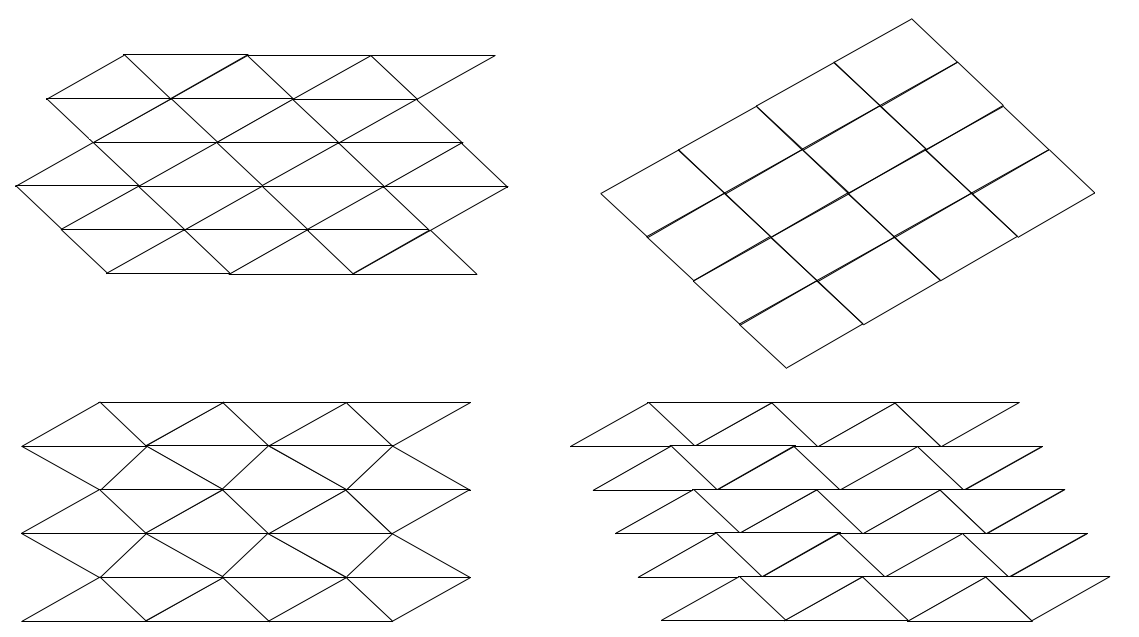}
\caption{A triangle tiling of the kind we study, and three related tilings}
\label{triangle-tilings}
\end{figure}

%
%

{\bf The trihexagonal tiling:} This tiling exhibits a surprising level of instability. We were unable to find a single stable trajectory in the tiling; during experimentation, even a small change in the direction or starting location of a trajectory would produce wildly different behavior. This is in marked contrast to some triangle tilings, where even a large change in direction or starting location had almost no effect on the behavior. We would like to understand why the trihexagonal tiling is so unstable. We would also like to show that dense trajectories exist. Examples of periodic trajectories that we have found are in Figures \ref{bigperiods1}-\ref{bigperiods2}, and we conjecture that there are periodic trajectories of arbitrarily high period.

{\bf Further extensions:} 
Our work considered only simple tilings of the plane. It would be interesting to consider some more complex tilings, such as the Penrose tiling or random tilings. 

Our work also considered only polygonal tilings; perhaps tilings or divisions of the plane by curves would yield interesting dynamics. Inner billiards on curved tables, such as ellipses, yields beautiful mathematics, see \cite{sergei}, so tiling billiards with curves are another possible direction.


Additionally, since the motivation for our work is refraction of light through solids, an obvious next step would be to study this problem in three dimensions, since any real-world application to creating perfect lenses or invisibility shields would likely require solid materials.

\newpage
\subsection{Acknowledgements}

This research was conducted during the Summer@ICERM program in 2013, where the first author was a teaching assistant and the second, third and fourth authors were undergraduate researchers. We are grateful to ICERM for excellent working conditions and the inspiring atmosphere. We thank faculty advisors Sergei Tabachnikov and Chaim Goodman-Strauss for their guidance. We also thank Pat Hooper for sharing his Java code, which we used to model this system.

The two reviewers made suggestions that substantially improved this paper. We especially thank one of them for suggesting the elegant method of proof in Section \ref{alex} using compositions of reflections, which replaced our previous methods that used trigonometry and arithmetic.


\vspace{1in}

\noindent {\bf Diana Davis}, Mathematics Department, Northwestern University, 2033 Sheridan Road, Evanston IL 60208, \href{mailto:diana@math.northwestern.edu}{diana@math.northwestern.edu} \\

\noindent {\bf Kelsey DiPietro}, Department of Applied and Computational Mathematics and Statistics, University of Notre Dame, 153 Hurley Hall, Notre Dame, IN 46556,  \href{mailto:kdipiet1@nd.edu}{kdipiet1@nd.edu} \\

\noindent {\bf Jenny Rustad}, Department of Mathematics, University of Maryland, 4176 Campus Drive, College Park, MD 20742, \href{mailto:jrustad1@math.umd.edu}{jrustad1@math.umd.edu} \\

\noindent {\bf Alexander St Laurent}, Department of Mathematics and Department of Computer Science, Brown University, 151 Thayer Street, Providence, RI 02912, 
\href{mailto:alexander_st_laurent@brown.edu}{alexander\_st\_laurent@brown.edu}

\end{document}